\documentclass[largeformat,11pt]{interact}

\usepackage{epstopdf}% To incorporate .eps illustrations using PDFLaTeX, etc.
\usepackage[caption]{subfig}% Support for small, `sub' figures and tables

\usepackage[numbers,sort&compress]{natbib}% Citation support using natbib.sty
\bibpunct[, ]{[}{]}{,}{n}{,}{,}% Citation support using natbib.sty
\renewcommand\bibfont{\fontsize{10}{12}\selectfont}% Bibliography support using natbib.sty
\makeatletter% @ becomes a letter
\def\NAT@def@citea{\def\@citea{\NAT@separator}}% Suppress spaces between citations using natbib.sty
\makeatother% @ becomes a symbol again

\usepackage{algorithm}
\usepackage{color}
\usepackage{algorithmic}
\theoremstyle{plain}% Theorem-like structures provided by amsthm.sty
\newtheorem{theorem}{Theorem}[section]
\newtheorem{lemma}[theorem]{Lemma}
\newtheorem{corollary}[theorem]{Corollary}

\theoremstyle{definition}
\newtheorem{definition}[theorem]{Definition}
\newtheorem{assumption}{Assumption}
\newtheorem{example}[theorem]{Example}

\theoremstyle{remark}
\newtheorem{remark}{Remark}

\usepackage{url}
\usepackage{convex}
\usepackage{accents}
\begin{document}

%\articletype{ARTICLE TEMPLATE}% Specify the article type or omit as appropriate

\title{Parametric analysis of semidefinite optimization}

\author{
\name{Ali Mohammad-Nezhad\thanks{CONTACT Ali Mohammad-Nezhad. Email: alm413@lehigh.edu; Tel.: +1(610)758-2903.} and Tam\'as Terlaky}
\affil{\textsuperscript{a}Department of Industrial and Systems Engineering, Lehigh University, Harold S. Mohler Laboratory, 200 West Packer Avenue, Bethlehem, PA 18015-1582}
}

\maketitle

\begin{abstract}
In this paper, we study parametric analysis of semidefinite optimization problems w.r.t. the perturbation of the objective function. We study the behavior of the optimal partition and optimal set mapping on a so-called nonlinearity interval. Furthermore, we investigate the sensitivity of the approximation of the optimal partition in a nonlinearity interval, which has been recently studied by Mohammad-Nezhad and Terlaky. The approximation of the optimal partition was obtained from a bounded sequence of interior solutions on, or in a neighborhood of the central path. We derive an upper bound on the distance between the approximations of the optimal partitions of the original and perturbed problems. Finally, we examine the theoretical bounds by way of experimentation.
\end{abstract}


\begin{amscode}
90C51, 90C22, 90C25 
\end{amscode}

%%%%%%%%
%New Section
%%%%%%%%
\section{Introduction}\label{intro}
In this paper, we study parametric analysis of a semidefinite optimization (SDO) problem in which the objective function is perturbed along a fixed direction. Mathematically, we consider
 \begin{align*}
&(\mathrm{P_{\epsilon}}) \qquad \min\!\Big \{\langle C + \epsilon \bar{C},  X \rangle \mid \langle A^{i} , X \rangle=b_i, \quad  i=1,\ldots, m, \ X \succeq 0 \Big \},\\[-1\jot]
&(\mathrm{D_{\epsilon}}) \qquad \max\!\bigg \{b^Ty  \mid  \sum_{i=1}^m y_i A^i+S=C + \epsilon \bar{C}, \ S \succeq 0, \ y \in \mathbb{R}^m \bigg \},
 \end{align*}
where $C,X,A^i \in \mathbb{S}^n$ for $i=1,\ldots, m$, $b \in \mathbb{R}^m$, $\bar{C} \in \mathbb{S}^n$ is a fixed direction, and $\mathbb{S}^n$ denotes the vector space of symmetric $n \times n$ matrices endowed with the inner product $\langle C,X \rangle:=\trace(CX)$. The optimal value of $(\mathrm{P_{\epsilon}})$ yields the optimal value function $v:\mathbb{R} \to \mathbb{R}\cup\{-\infty,+\infty\}$. In this context, $X \succeq 0$ means that $X$ belongs to the cone of positive semidefinite matrices, which is denoted by $\mathbb{S}_+^{n}$. 

\vspace{5px}
\noindent
In order to guarantee zero duality gap and attainment of the optimal values, see e.g.,~\cite{BS00}, we make the following assumptions throughout this paper:
\begin{assumption}\label{ass1}
The coefficient matrices $A^1,\ldots,A^m$ are linearly independent.
\end{assumption}
\begin{assumption}\label{ass2}
The interior point condition holds at $\epsilon = 0$, i.e., there exists a strictly feasible solution $\big(X^{\circ}(0),y^{\circ}(0),S^{\circ}(0)\big)$ with $X^{\circ}(0),S^{\circ}(0) \succ 0$, where $\succ 0$ means positive definite.
 \end{assumption}

\vspace{5px}
\noindent
Let $\mathcal{E}\subseteq \mathbb{R}$ be the set of all $\epsilon$ for which $v(\epsilon) > -\infty$. By Assumption~\ref{ass2}, $\mathcal{E}$ is nonempty and non-singleton, since both $(\mathrm{P_{\epsilon}})$ and $(\mathrm{D_{\epsilon}})$ have strictly feasible solutions for all $\epsilon$ in a sufficiently small neighborhood of $0$. Further, $v(\epsilon)$ is a proper concave function on $\mathcal{E}$, and $\mathcal{E}$ is a closed, possibly unbounded, interval, see e.g., Lemma 2.2 in~\cite{BJRT96}. The continuity of $v(\epsilon)$ on $\interior(\mathcal{E})$ follows from its concavity on $\mathcal{E}$, see Corollary 2.109 in~\cite{BS00}.

\vspace{5px}
\noindent
The primal and dual optimal set mappings are defined as
\begin{align*}
&\mathcal{P}^*(\epsilon):=\big\{X \mid \langle C + \epsilon \bar{C}, X \rangle = v(\epsilon), \ \langle A^{i} , X \rangle=b_i, \ \  i=1,\ldots, m, \ X \succeq 0 \big\}, \\[-1\jot]
&\mathcal{D}^*(\epsilon):=\big \{(y,S) \mid b^T y = v(\epsilon), \ \sum_{i=1}^m y_i A^i+S=C + \epsilon \bar{C}, \ S \succeq 0\big \}.
\end{align*}
\noindent
Our analysis relies on the existence of central path and maximally complementary solutions~\cite{Kl02}, as formally defined below.
\begin{definition}
We call an optimal solution $\big(X^*(\epsilon),y^*(\epsilon),S^*(\epsilon)\big)$ maximally complementary if 
\begin{align*}
X^*(\epsilon) \in \ri\big(\mathcal{P}^*(\epsilon)\big), \quad \text{and} \quad \big(y^*(\epsilon),S^*(\epsilon)\big) \in \ri\big(\mathcal{D}^*(\epsilon)\big),
\end{align*}
where $\ri(.)$ denotes the relative interior of a convex set. An optimal solution $\big(X^*(\epsilon),y^*(\epsilon),S^*(\epsilon)\big)$ is called strictly complementary if
\begin{align*}
X^*(\epsilon)+S^*(\epsilon) \succ 0.
\end{align*}
\end{definition}
\noindent
As a result of a theorem of the alternative~\cite{CSW13}, Assumption~\ref{ass2} implies the interior point condition at every $\epsilon \in \interior(\mathcal{E})$, see Lemma 3.1 in~\cite{GS99}. Therefore, strong duality%
\footnote{Strong duality in this paper means that both the primal and dual problems admit optimal solutions with equal objective values.} holds, and both $\mathcal{P}^*(\epsilon)$ and $\mathcal{D}^*(\epsilon)$ are nonempty and compact for all $\epsilon \in \interior(\mathcal{E})$, see e.g., Theorem 5.81 in~\cite{BS00}. Consequently, for every $\epsilon \in \interior(\mathcal{E})$ there exists a maximally complementary solution, and an optimal solution $\big(X(\epsilon),y(\epsilon),S(\epsilon)\big)$ satisfies the complementarity condition $X(\epsilon) S(\epsilon)=0$. It is known that for an SDO problem, and in general a linear conic optimization problem, there might be no strictly complementary solution. 

%%%%%%%%
%New Section
%%%%%%%%
\subsection{Related works}
Steady advances in computational optimization have enabled us to solve a wide variety of SDO problems in polynomial time. Nevertheless, sensitivity analysis tools are still the missing parts of SDO solvers, e.g., interior point methods (IPMs) in SeDuMi~\cite{St99}, SDPT3~\cite{TTT99,TTT03}, and MOSEK%
\footnote{\url{https://www.mosek.com/}}. Shapiro~\cite{S97} established the differentiability of the optimal solution for a nonlinear SDO problem using the standard implicit function theorem. Under linear perturbations in the objective vector, the coefficient matrix, and the right hand side, Nayakkankuppam and Overton~\cite{NO99} derived the region of stability around an optimal solution of SDO which satisfies the strict complementarity and nondegeneracy conditions. The sensitivity of central solutions for SDO was considered in~\cite{FN01,St01}. Based on IPMs, Yildirim and Todd~\cite{Y2001} proposed a sensitivity analysis approach for linear optimization (LO) and SDO. Recently, Cheung and Wolkowicz~\cite{CW14} and Sekiguchi and Waki~\cite{SW16} studied the continuity of the optimal value function for SDO problems, which fail the interior point condition. A comprehensive study on sensitivity and stability of nonlinear optimization problems was given by Bonnans and Shapiro~\cite{BS00}. The results  are mostly valid in a neighborhood of a given optimal solution, and they depend on strong second-order sufficient conditions. We refer the reader to~\cite{Fi83} for a survey of classical results. 

\vspace{5px}
\noindent
 Adler and Monteiro~\cite{AM92} studied the parametric analysis of LO problems using the concept of optimal partition. Another treatment of sensitivity analysis for LO based on the optimal partition approach was given by Jansen et al.~\cite{JRT93} and Greenberg~\cite{G94}. Berkelaar et al.~\cite{BJRT96,BRT97} extended the optimal partition approach to linearly constrained quadratic optimization (LCQO) with perturbation in the right hand side vector and showed that the optimal value function is convex and piecewise quadratic. There have been further studies on optimal partition and parametric analysis of conic optimization problems. In contrast to LO, the optimal partition of SDO is defined as a 3-tuple of mutually orthogonal subspaces of $\mathbb{R}^n$, see Section~\ref{optimal_partition_definition}. Goldfarb and Scheinberg~\cite{GS99} considered a parametric SDO problem, where the objective is perturbed along a fixed direction. They derived auxiliary problems to compute the directional derivatives of the optimal value function and the so-called invariancy set of the optimal partition. Yildirim~\cite{Y2004} extended the concept of the optimal partition and the auxiliary problems in~\cite{GS99} for linear conic optimization problems. 

%%%%%%%%
%New Section
%%%%%%%%
\subsection{Contributions}
To the best of our knowledge, for the past twenty years, the regularity and stability of the trajectory of the optimal set mapping for conic optimization problems have received very little attention. In this paper, we take an initial step to fill this gap by revisiting parametric analysis of SDO problems. Our main goal is to elaborate on the optimal partition approach given in~\cite{GS99} for a parametric SDO problem. To that end, we introduce the concepts of nonlinearity interval and transition point for the optimal partition, and provide sufficient conditions for the existence of a nonlinearity interval and a transition point, see Theorems~\ref{nonlinearity_continuity} and~\ref{nonsingularity_sufficient}. On a nonlinearity interval, the rank of both $X^*(\epsilon)$ and $S^*(\epsilon)$ are invariant w.r.t. $\epsilon$, while a transition point is a singleton invariancy set, which does not belong to a nonlinearity interval. Further, we quantify the sensitivity of the approximation of the optimal partition given in~\cite{MT19}, see Theorems~\ref{Stewart_eigenspace_theorem} and~\ref{existence_Q_central_solutions}. In~\cite{MT19}, the approximation of the optimal partition was obtained from the eigenvectors of interior solutions on, or in a neighborhood of the central path. Such interior solutions are usually generated by a feasible primal-dual IPM, whose accumulation points belong to the relative interior of the optimal set. We conclude the paper with numerical experiments and directions for future research.

\vspace{5px}
\noindent
Roughly speaking, our main contributions are
\begin{itemize}
\item Theoretical characterization of nonlinearity intervals and transition points of the optimal partition;
\item Upper bounds for the sensitivity of the approximation of the optimal partition in a nonlinearity interval.
\end{itemize}

\vspace{5px}
\noindent
Along with an invariancy set, a nonlinearity interval can be construed as a stability region and its identification has a high impact on the post-optimal analysis of SDO problems, see e.g.,~\cite{CAPT17,CHS18}, in which the stability region of rank-one primal optimal solutions is of interest. Interestingly, the optimal value function for SDO problems has been shown to be piecewise algebraic~\cite{N10}, i.e., for each piece there exists a polynomial function $\Psi(.,.)$ so that $\Psi(v(\epsilon),\epsilon) = 0$. Further characterization of nonlinear pieces of $v(\epsilon)$ can be obtained by first identifying the nonlinearity intervals of the optimal partition.

%%%%%%%%
%New Section
%%%%%%%%
\subsection{Organization of the paper}\label{organization}
The rest of this paper is organized as follows. In Section~\ref{optimal_partition}, we concisely review the concepts of the nondegeneracy and optimal partition for SDO. In Section~\ref{sensitivity_partition}, we review the concept of an invariancy set from~\cite{GS99} and prove additional results. Furthermore, we introduce the notions of a nonlinearity interval and a transition point and present the behavior of the optimal partition on a nonlinearity interval. In Section~\ref{perturbation_optimal_partition}, we investigate the sensitivity of the approximation of the optimal partition in a nonlinearity interval. In Section~\ref{experiments}, we provide numerical experiments to illustrate the numerical behavior of the central solutions and the optimal partition w.r.t. $\epsilon$. Our concluding remarks and topics for future studies are stated in Section~\ref{conclusion}.

\vspace{5px}
\noindent
\textbf{Notation:} Throughout this paper, for any given $\epsilon \in \interior(\mathcal{E})$, $\big(X(\epsilon),y(\epsilon),S(\epsilon)\big)$ denotes a primal-dual optimal solution, and a maximally complementary solution is denoted by $\big(X^*(\epsilon),y^*(\epsilon),S^*(\epsilon)\big)$. For a given matrix, $\mathcal{R}(.)$ stands for its column space, and $\Null(.)$ represents its null space. Given a symmetric matrix $X$, $\svectorize:\mathbb{S}^{n} \to \mathbb{R}^{n(n+1)/2}$ is a linear transformation, which multiplies the off-diagonal entries of a symmetric matrix by $\sqrt{2}$, and stacks the upper triangular part into a vector, i.e.,
\begin{align}
\svectorize(X):=\big(X_{11}, \sqrt{2}X_{12},\ldots, \sqrt{2}X_{1n}, X_{22}, \sqrt{2} X_{23},\ldots,\sqrt{2}X_{2n},\ldots, X_{nn}\big)^T, \label{linear_transformation_svec}
\end{align}
and $\svecinv\!:\mathbb{R}^{n(n+1)/2} \to \mathbb{S}^{n}$ is the inverse of $\svectorize(.)$. Furthermore, $\lambda_{[i]}(X)$ denotes the $i^{\mathrm{th}}$ largest eigenvalue of $X$. Thus, $\lambda_{\max}(X):=\lambda_{[1]}(X)$, $\lambda_{\min}(X):=\lambda_{[n]}(X)$, and $\Lambda(X)$ denotes the diagonal matrix of the eigenvalues of $X \in \mathbb{S}^n$. The Frobenius norm is denoted by $\|.\|$, and the $l_2$ norm and the induced $2$-norm (spectral norm) for the vectors and matrices are indicated by $\|.\|_2$. By $\distance(\mathcal{S}_1,\mathcal{S}_2)$ we mean the distance between two subspaces $\mathcal{S}_1$ and $\mathcal{S}_2$ of $\mathbb{R}^n$ with the same dimension, which is defined as
\begin{align}\label{metric}
\distance(\mathcal{S}_1,\mathcal{S}_2):=\big \|\proj_{\mathcal{S}_1}-\proj_{\mathcal{S}_2} \big \|_2,
\end{align} 
where $\proj_{\mathcal{S}_1}$ and $\proj_{\mathcal{S}_2}$ are the orthogonal projections onto the subspaces $\mathcal{S}_1$ and $\mathcal{S}_2$, respectively, see Section 2.5.3 in~\cite{GL13}. The notation $(.;.;\ldots;.)$ and $(.,.,\ldots,.)$ is adopted for the concatenation and side by side arrangement of column vectors, respectively. 
%%%%%%%%
%New Section
%%%%%%%%
\section{Preliminaries}\label{optimal_partition}
For the sake of simplicity, we assume that $\epsilon = 0$, and define $(\mathrm{P}):=(\mathrm{P_{0}})$ and $(\mathrm{D}):=(\mathrm{D_{0}})$. We adopt this notation for an optimal solution, a maximally complementary solution, and central solutions as well. 

%%%%%%%%
%New Section
%%%%%%%%
\subsection{Nondegeneracy conditions}\label{Prima_dual_nondeg} 
The primal and dual nondegeneracy conditions were studied for SDO in~\cite{AHO97}. Let $(X,y,S) \in \mathcal{P} \times \mathcal{D}$ be a primal-dual feasible solution. Consider the eigendecompositions 
\begin{align*}
X=M\Lambda(X)M^T, \qquad S=N\Lambda(S)N^T,
\end{align*}
where $M:=(M_1,M_2)$ and $N:=(N_1,N_2)$ are orthogonal matrices, and $M_1$ and $N_2$ correspond to the positive eigenvalues of $X$ and $S$, respectively. A primal feasible solution $X$ is called primal nondegenerate if the matrices
\begin{align*}
\begin{pmatrix} M^T_1 A^i M_1 & M^T_1 A^i M_2\\M^T_2 A^i M_1 & 0 \end{pmatrix} 
\end{align*}
for $i=1,\ldots,m$ are linearly independent in $\mathbb{S}^n$. A dual feasible solution $(y,S)$ is called dual nondegenerate if the matrices $N^T_1 A^i N_1$ for $i=1,\ldots,m$ span $\mathbb{S}^{n-\rank(S)}$. If there exists a primal nondegenerate (dual nondegenerate) optimal solution, then the dual (primal) optimal solution is unique. In case that the strict complementarity condition holds, then a unique primal (dual) optimal solution implies the existence of a dual (primal) nondegenerate optimal solution. The proofs can be found in~\cite{AHO97,Kl02}.
%%%%%%%%
%New Section
%%%%%%%%  
\subsection{Optimal partition of SDO}\label{optimal_partition_definition}
It is known~\cite{Kl02} that the optimal partition is well-defined under the interior point condition. Let $(X^*,y^*,S^*) \in \ri\big(\mathcal{P}^* \times \mathcal{D}^*\big)$ be a maximally complementary solution, where $X^*$ and $S^*$ have a common eigenvector basis due to the complementarity condition $X^*S^*=0$. The subspaces $\mathcal{R}(X^*)$ and $\mathcal{R}(S^*)$ are orthogonal by the complementarity condition, and those subspaces are spanned by the eigenvectors corresponding to the positive eigenvalues of $X^*$ and $S^*$, respectively. Let $\mathcal{B}:=\mathcal{R}(X^*)$, $\mathcal{N}:=\mathcal{R}(S^*)$, and $\mathcal{T}:=\big(\mathcal{R}(X^*) + \mathcal{R}(S^*)\big)^{\perp}$, where $\perp$ denotes the orthogonal complement of a subspace. Then $(\mathcal{B}, \mathcal{T},\mathcal{N})$ is called the \textit{optimal partition} of $(\mathrm{P})$ and $(\mathrm{D})$. Note that $\mathcal{T}=\{0\}$ if and only if a strictly complementary solution exists. Since $(X^*,y^*,S^*)$ has the highest rank in $\mathcal{P}^* \times \mathcal{D}^*$, see e.g., Lemma 2.3 in~\cite{Kl02}, the optimal partition is uniquely defined. We use the notation $Q:=\big(Q_{\mathcal{B}},Q_{\mathcal{T}}, Q_{\mathcal{N}}\big)$ to denote an orthonormal basis partitioned according to the subspaces $\mathcal{B}$, $\mathcal{T}$, and $\mathcal{N}$. 

\vspace{5px}
\noindent
Let us define $n_{\mathcal{B}}:=\dim\!\big(\mathcal{B}\big)$, $n_{\mathcal{N}}:=\dim\!\big(\mathcal{N}\big)$, and $n_{\mathcal{T}}:=\dim\!\big(\mathcal{T}\big)$. By the interior point condition, at least one of $n_{\mathcal{B}}$ or $n_{\mathcal{N}}$ has to be positive, see Remark 2 in~\cite{MT19}.
%%%%%%%%
%New Theorem
%%%%%%%% 
\begin{theorem}[Theorem 2.7 in~\cite{Kl02}]\label{optimal_sol_rep}
Any $\big(\check{X},\check{y},\check{S}\big) \in \mathcal{P}^* \times \mathcal{D}^*$ can be represented as
\begin{align*}
\check{X}=Q_{\mathcal{B}} U_{\check{X}}Q^T_{\mathcal{B}}, \qquad \check{S}=Q_{\mathcal{N}} U_{\check{S}} Q^T_{\mathcal{N}},
\end{align*}
with unique $U_{\check{X}} \succeq 0 $ and $U_{\check{S}} \succeq 0$. If $n_{\mathcal{B}} > 0$ and $\check{X} \in \ri(\mathcal{P}^*)$, then we have $U_{\check{X}} \succ 0$. Analogously, if $n_{\mathcal{N}} > 0$ and $\big(\check{y},\check{S}\big) \in  \ri\big(\mathcal{D}^*\big)$, then we have $U_{\check{S}} \succ 0$. \qed
\end{theorem}
%
%%%%%%%%
%New Section
%%%%%%%%
\subsection{Approximation of the optimal partition}\label{approximation_optimal_partition}
The central path is a smooth trajectory of solutions to
\begin{equation}\label{central_path_equations}
\begin{aligned}
\langle A^i , X \rangle&=b_i, \qquad   i=1,\ldots, m, & X &\succeq 0,\\[-1\jot]
\sum_{i=1}^m A^i y_i+S&=C, & S &\succeq 0,\\[-1\jot]
XS&=\mu I_n,
\end{aligned}
\end{equation}
where $I_n$ denotes the $n \times n$ identity matrix. For a given $\mu>0$, the unique solution of~\eqref{central_path_equations}, denoted by $\big(X^{\mu},y^{\mu},S^{\mu}\big)$, is called a central solution. The existence and uniqueness follow from Assumptions~\ref{ass1} and~\ref{ass2}, see Theorem 3.1 in~\cite{Kl02}. Furthermore, it is immediate from the nonlinear equations $XS=\mu I_n$ that $X^{\mu}$ and $S^{\mu}$ have a common eigenvector basis for every $\mu > 0$. 

\vspace{5px}
\noindent
The derivation of bounds in~\cite{MT19} for the vanishing blocks of $Q^T X^{\mu}Q$ and $Q^T S^{\mu}Q$ is based on a condition number $\sigma$ defined as 
\begin{align*}
   \sigma_{\mathcal{B}} &:=\begin{cases} \max\limits_{\check{X} \in \mathcal{P}^*} \lambda_{\min}\big(Q^T_{\mathcal{B}} \check{X}Q_{\mathcal{B}}\big),  \quad \ &\mathcal{B} \neq \{0\},\\ \infty,   &\mathcal{B}=\{0\}, \end{cases} \\
   \sigma_{\mathcal{N}} &:=\begin{cases} \max\limits_{(\check{y},\check{S}) \in \mathcal{D}^*} \lambda_{\min}\big(Q_{\mathcal{N}}^T\check{S}Q_{\mathcal{N}}\big), &\mathcal{N} \neq\{0\}, \\ \infty, &\mathcal{N} = \{0\}, \end{cases} \\
   \sigma &:= \min\{\sigma_{\mathcal{B}},\sigma_{\mathcal{N}}\}.
\end{align*}
%%%%%%%%
%New Remark
%%%%%%%%
\begin{remark}
Due to the interior point condition and the compactness of the optimal set, $\sigma$ is well-defined, see Lemma 3.1 in~\cite{MT19}. 
\end{remark}
\noindent
Furthermore, we need an error bound result for the following linear matrix inequality (LMI) systems 
\begin{equation}
\label{KKT_equivalent}
\begin{cases}
\mathcal{A}x=b,\\
 \svectorize(\check{S})^Tx = 0,\\
\svecinv(x) \succeq 0,
\end{cases}
\qquad
\begin{cases}
\mathcal{A}^Ty + s=\svectorize(C),\\
\svectorize(\check{X})^T s = 0,\\
\svecinv(s) \succeq 0,
\end{cases}
\end{equation}
where
 \begin{align}
\mathcal{A}:=\big(\svectorize(A^1),\ldots,\svectorize(A^m)\big)^T, \label{coefficient_transformation}
\end{align}
and $\big(\check{X},\check{y},\check{S}\big) \in \mathcal{P}^* \times \mathcal{D}^*$ is a primal-dual optimal solution. The LMIs given in~\eqref{KKT_equivalent} indeed define the set of primal and dual optimal solutions.  

\vspace{5px}
\noindent
By the orthogonal projection of $\svectorize\big(X^{\mu}\big)$ and $\svectorize\big(S^{\mu}\big)$ onto the affine subspaces 
\begin{align*}
\bar{\mathcal{L}}_{\mathcal{P}}&:=\!\Big \{x \in \mathbb{R}^{n(n+1)/2} \mid x \in \svectorize(\check{X}) + \Null(\mathcal{A}), \ \svectorize(\check{S})^Tx = 0 \Big \},\\
\bar{\mathcal{L}}_{\mathcal{D}}&:=\!\Big\{s \in  \mathbb{R}^{n(n+1)/2} \mid s \in \svectorize(\check{S}) + \mathcal{R}\big(\mathcal{A}^T\big), \ \svectorize(\check{X})^T s = 0 \Big \}, 
\end{align*}
 we get
\begin{align}
\distance\!\big(\svectorize\!\big(X^{\mu}\big),\bar{\mathcal{L}}_{\mathcal{P}}\big) \le \theta_1 n\mu,\quad \text{and} \quad \distance\!\big(\svectorize\!\big(S^{\mu}\big),\bar{\mathcal{L}}_{\mathcal{D}}\big) \le \theta_2n\mu, \label{Hoffman_upper_bound}
\end{align}
in which 
\begin{align*}
\distance\!\big(\svectorize\!\big(X^{\mu}\big), \bar{\mathcal{L}}_{\mathcal{P}} \big ):=\min\!\Big\{\big \|x - \svectorize\big(X^{\mu}\big)\big \|_2 \mid \mathcal{A}x=b, \  \svectorize(\check{S})^Tx = 0\Big\}
\end{align*}
stands for the distance of $\svectorize\big(X^{\mu}\big)$ from the affine subspace $\bar{\mathcal{L}}_{\mathcal{P}}$, and the condition numbers $\theta_1$ and $\theta_2$ depend on $\mathcal{A}$ and $\check{S}$, and $\mathcal{A}$ and $\check{X}$, respectively%
\footnote{The upper bounds in~\eqref{Hoffman_upper_bound} can be simply obtained by minimizing $\big\|x-\svectorize\big(X^{\mu}\big)\big\|^2_2$ and $\big\|s-\svectorize\big(S^{\mu}\big)\big\|^2_2$ on the affine subspaces $\bar{\mathcal{L}}_{\mathcal{P}}$ and $\bar{\mathcal{L}}_{\mathcal{D}}$, respectively, and using Lagrange multipliers method.}, see Section 3.2 in~\cite{MT19}. Interestingly, $\theta_1$ and $\theta_2$ can be considered as Hoffman~\cite{Hoff52} condition numbers. As a consequence, if 
\begin{align}
\mu \le \hat{\mu}:=\frac 1n \min\!\Big \{\theta_1^{-1},\theta_2^{-1} \Big \}, \label{mu_hat}
\end{align}
then it holds that 
\begin{align*}
\distance\!\big(\svectorize\!\big(X^{\mu}\big),\bar{\mathcal{L}}_{\mathcal{P}}\big) \le 1, \qquad \distance\!\big(\svectorize\!\big(S^{\mu}\big),\bar{\mathcal{L}}_{\mathcal{D}}\big) \le 1. 
 \end{align*}
 Lemma~\ref{Holderian_error_bound} is in order.
 %%%%%%%
%New Lemma
 %%%%%%%
\begin{lemma}[Lemma 3.5 in~\cite{MT19}]\label{Holderian_error_bound}
Let a central solution $\big (X^{\mu},y^{\mu},S^{\mu} \big )$ be given, where $\mu$ satisfies~\eqref{mu_hat}. Then there exist a positive condition number $c$ independent of $\mu$ and a positive exponent $\gamma \ge 2^{1-n}$ such that
\begin{equation}\label{error_bounds}
\begin{aligned}
\distance\!\big(X^{\mu}, \svecinv\!\big(\bar{\mathcal{L}}_{\mathcal{P}}\big) \cap \mathbb{S}^n_+\big) \le c(n\mu)^{\gamma}, 
\qquad \distance\!\big(S^{\mu}, \svecinv\!\big(\bar{\mathcal{L}}_{\mathcal{D}}\big) \cap \mathbb{S}^n_+\big) \le c(n\mu)^{\gamma}.
\end{aligned}
\end{equation}
\end{lemma}
%%%%%%%%
%New Remark
%%%%%%%%
\begin{remark}
The reader is referred to Section 3.2 in~\cite{MT19} for further discussion of the exponent $\gamma$ and the condition number $c$.
\end{remark}

\vspace{5px}
\noindent
Using the condition number $\sigma$ and the error bounds in~\eqref{error_bounds}, Lemma~\ref{bounds_on_muCenters} specifies lower and upper bounds on the eigenvalues of $X^{\mu}$ and $S^{\mu}$. 
 %%%%%%%
%New Lemma
 %%%%%%%
\begin{lemma}[Theorem 3.8 in~\cite{MT19}]\label{bounds_on_muCenters}
Let $\big(X^{\mu},y^{\mu},S^{\mu}\big)$ be given, where $\mu \le \hat{\mu}$. Then we have
\begin{align}
\lambda_{[n-i+1]}\big(S^{\mu}\big) &\le \frac{n \mu}{\sigma}, &  \lambda_{[i]}\big(X^{\mu}\big) &\ge \frac{\sigma}{n}, & i&=1,\ldots, n_{\mathcal{B}}, \label{B_partition}\\
 \lambda_{[n-i+1]}\big(X^{\mu}\big) &\le \frac{n \mu}{\sigma}, & \lambda_{[i]}\big(S^{\mu}\big) &\ge \frac{\sigma}{n},  & i&=1,\ldots, n_{\mathcal{N}},  \label{N_partition}\\
\frac{\mu}{c\sqrt{n}(n \mu)^{\gamma}} &\le   \lambda_{[i]}\big(X^{\mu}\big),\mkern-14mu &\lambda_{[n-i+1]}\big(S^{\mu}\big) &\le c\sqrt{n}(n \mu)^{\gamma}, &   i&=n_{\mathcal{B}}+1,\ldots, n_{\mathcal{B}}+n_{\mathcal{T}}. \label{T_cardinal}
\end{align}
\end{lemma}
\noindent
Consider the eigendecompositions $X^{\mu}=Q^{\mu} \Lambda\big(X^{\mu}\big) (Q^{\mu})^T$ and $S^{\mu}=Q^{\mu} \Lambda\big(S^{\mu}\big) (Q^{\mu})^T$, in which $Q^{\mu}$ is a common orthonormal eigenvector basis. One can observe  from~\eqref{B_partition} to~\eqref{T_cardinal} that if $\mu$ is so small that the intervals $\big[\mu/(c\sqrt{n}(n \mu)^{\gamma}), \ c\sqrt{n}(n \mu)^{\gamma} \big]$, $(0,n\mu/\sigma]$, and $[\sigma/n, \infty)$ do not overlap, then we can partition the columns of $Q^{\mu}$ into $Q^{\mu}_{\mathcal{B}}$, $Q^{\mu}_{\mathcal{T}}$, and $Q^{\mu}_{\mathcal{N}}$, where the accumulation points of $Q^{\mu}_{\mathcal{B}}$, $Q^{\mu}_{\mathcal{T}}$, and $Q^{\mu}_{\mathcal{N}}$ form orthonormal bases for $\mathcal{B}$, $\mathcal{T}$, and $\mathcal{N}$, see also Remark~\ref{convergence_subspaces}. This holds if $\mu$ satisfies
\begin{align*}
\frac{\mu}{c\sqrt{n}(n \mu)^{\gamma}} \le c\sqrt{n}(n\mu)^{\gamma},  \quad \frac{n\mu}{\sigma} <  \frac{\mu}{c\sqrt{n}(n \mu)^{\gamma}}, \quad c\sqrt{n}(n\mu)^{\gamma} < \frac{\sigma}{n}, \quad \frac{n\mu}{\sigma} < \frac{\sigma}{n}, 
\end{align*}
or equivalently
\begin{align}\label{upper_bound_partition}
\mu < \tilde{\mu}:=\min\!\bigg\{\frac{1}{n} \bigg (\frac{\sigma}{cn^{\frac32}} \bigg)^{\frac{1}{\gamma}}, \ \frac{\sigma^2}{n^2}, \ \hat{\mu}\bigg\}.
\end{align}
%%%%%%%%
%New Remark
%%%%%%%%
\begin{remark}\label{rem:optimal_partition_approximations}
For a fixed $\mu$ with $\mu < \tilde{\mu}$, we refer to $\mathcal{R}\big(Q^{\mu}_{\mathcal{B}}\big)$, $\mathcal{R}\big(Q^{\mu}_{\mathcal{T}}\big)$, and $\mathcal{R}\big(Q^{\mu}_{\mathcal{N}}\big)$ as approximations of $\mathcal{B}$, $\mathcal{T}$, and $\mathcal{N}$, respectively.
\end{remark}
%%%%%%%%
%New Remark
%%%%%%%%
\begin{remark}
The identification of $\big(Q_{\mathcal{B}}^{\mu},Q_{\mathcal{T}}^{\mu},Q_{\mathcal{N}}^{\mu}\big)$ without the explicit knowledge of the condition numbers is discussed in~\cite{MT19,MT19a}.
\end{remark}

%%%%%%%%
%New Section
%%%%%%%%
\section{Sensitivity of the optimal partition}\label{sensitivity_partition}
In this section, we investigate the behavior of the optimal partition and the optimal set mapping under perturbation of the objective vector. From now on, 
\begin{align*}
\pi(\epsilon):=\big(\mathcal{B}(\epsilon),\mathcal{T}(\epsilon),\mathcal{N}(\epsilon)\big)
\end{align*}
denotes the optimal partition of $(\mathrm{P_{\epsilon}})$ and $(\mathrm{D_{\epsilon}})$ for a given $\epsilon$ and 
\begin{align*}
Q_{\epsilon}:=\big(Q_{\mathcal{B}(\epsilon)},Q_{\mathcal{T}(\epsilon)}, Q_{\mathcal{N}(\epsilon)}\big)
\end{align*}
denotes an orthonormal basis partitioned according to the subspaces $\mathcal{B}(\epsilon)$, $\mathcal{T}(\epsilon)$, and $\mathcal{N}(\epsilon)$. We introduce and characterize the subintervals of $\interior(\mathcal{E})$ on which the optimal partition or the dimensions of both $\mathcal{B}(\epsilon)$ and $\mathcal{N}(\epsilon)$ are invariant of $\epsilon$. The discussion is motived by minimizing a parametric objective function on the 3-elliptope:
 \begin{figure}[]
\begin{center}
\includegraphics[height=2in]{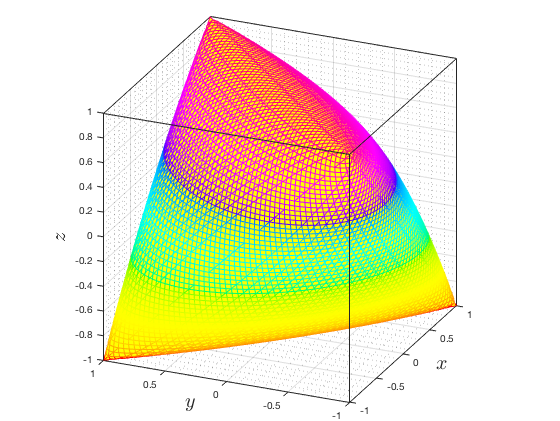}
\caption{The illustration of a 3-elliptope.}
\label{fig:3elliptope}
\end{center}
\end{figure}
\begin{align}
\mathcal{E}\ell\ell_3:=\Bigg\{(x,y,z)\in\mathbb{R}^3 \Bigg | \begin{pmatrix} 1 & \ x & \ y\\x & \ 1 & \ z\\y & \ z & \ 1 \end{pmatrix} \succeq 0\Bigg\}. \label{feasible_3elliptope}
\end{align}

\begin{example}\label{motivation_nonlinearity}
Consider the following SDO problem:
\begin{align*}
A^1&=\begin{pmatrix} 1 & 0 & 0\\0 & 0  & 0\\0 & 0 & 0\end{pmatrix}, & A^2&=\begin{pmatrix} 0 & 0 & 0\\0 & 1  & 0\\0 & 0 & 0\end{pmatrix}, & A^3&=\begin{pmatrix} 0 & 0 & 0\\0 & 0  & 0\\0 & 0 & 1\end{pmatrix}, \\
C&=\begin{pmatrix} \ \ 0 & -1 & \ \ 1\\ -1 & \ \ 0  & -1\\ \ \ 1 & -1 & \ \ 0\end{pmatrix}, & \bar{C}&=\begin{pmatrix} \ \ 0 & \ 2 & -2\\ \ \ 2 & \ 0  & \ \ 0\\-2 & \ 0 & \ \ 0\end{pmatrix}, & b&=(1, \ 1, \ 1)^T,
\end{align*}
where $\mathcal{E}=\mathbb{R}$, since the feasible set~\eqref{feasible_3elliptope} is compact as illustrated in Figure~\ref{fig:3elliptope}. 
\noindent
The optimal partition at $\epsilon$ is given by
\begin{align*}
\mathcal{B}(\epsilon)&=\mathcal{R}\Bigg(\begin{pmatrix} -\frac{1}{\sqrt{3}} \\ -\frac{1}{\sqrt{3}} \\ \ \ \frac{1}{\sqrt{3}} \end{pmatrix} \Bigg), &
\mathcal{T}(\epsilon)&=\{0\}, &
\mathcal{N}(\epsilon)&=\mathcal{R}\Bigg(\begin{pmatrix} 0 & \ \ \frac{2}{\sqrt{6}} \\ \frac{1}{\sqrt{2}} & -\frac{1}{\sqrt{6}}\\ \frac{1}{\sqrt{2}} & \ \ \frac{1}{\sqrt{6}} \end{pmatrix}\Bigg), & \epsilon &< -\frac12,\\
\mathcal{B}(\epsilon)&=\mathcal{R}\Bigg(\begin{pmatrix} -\frac{1}{\sqrt{3}} \\ -\frac{1}{\sqrt{3}} \\ \ \ \frac{1}{\sqrt{3}} \end{pmatrix} \Bigg), &
\mathcal{T}(\epsilon)&=\mathcal{R}\Bigg(\begin{pmatrix} 0  \\ \frac{1}{\sqrt{2}} \\ \frac{1}{\sqrt{2}} \end{pmatrix}\Bigg), &
\mathcal{N}(\epsilon)&=\mathcal{R}\Bigg(\begin{pmatrix} \ \ \frac{2}{\sqrt{6}} \\ -\frac{1}{\sqrt{6}} \\  \ \ \frac{1}{\sqrt{6}} \end{pmatrix}\Bigg), & \epsilon &= -\frac12,\\
\mathcal{B}(\epsilon)&=\mathcal{R}\Bigg(\begin{pmatrix} \ \ \frac{1}{\sqrt{3}} \\ -\frac{1}{\sqrt{3}} \\ \ \ \frac{1}{\sqrt{3}} \end{pmatrix} \Bigg), &
\mathcal{T}(\epsilon)&=\mathcal{R}\Bigg(\begin{pmatrix} 0  \\ \frac{1}{\sqrt{2}} \\ \frac{1}{\sqrt{2}} \end{pmatrix}\Bigg), &
\mathcal{N}(\epsilon)&=\mathcal{R}\Bigg(\begin{pmatrix} -\frac{2}{\sqrt{6}} \\ -\frac{1}{\sqrt{6}}\\ \ \ \frac{1}{\sqrt{6}} \end{pmatrix}\Bigg), & \epsilon &= \frac32,\\
\mathcal{B}(\epsilon)&=\mathcal{R}\Bigg(\begin{pmatrix} \ \ \frac{1}{\sqrt{3}} \\ -\frac{1}{\sqrt{3}} \\ \ \ \frac{1}{\sqrt{3}} \end{pmatrix} \Bigg), &
\mathcal{T}(\epsilon)&=\{0\}, &
\mathcal{N}(\epsilon)&=\mathcal{R}\Bigg(\begin{pmatrix} 0 &  -\frac{2}{\sqrt{6}} \\ \frac{1}{\sqrt{2}} & -\frac{1}{\sqrt{6}}\\ \frac{1}{\sqrt{2}} & \ \ \frac{1}{\sqrt{6}} \end{pmatrix}\Bigg), & \epsilon &> \frac32, 
\end{align*}
while for all $\epsilon \in (-\frac12,\frac32)$ we have
\begin{equation}\label{optimal_partition_nonlinearity}
\begin{aligned}
\begingroup 
\setlength\arraycolsep{1pt}
\mathcal{B}(\epsilon)=\mathcal{R}\Bigg(\begin{pmatrix} 0 & \lim_{\epsilon' \downarrow \epsilon}  \frac{2\mathrm{sgn}(2\epsilon'-1)}{\sqrt{2(2\epsilon'-1)^2+4}}\\\frac{1}{\sqrt{2}} & \frac{-|2\epsilon-1|}{\sqrt{2(2\epsilon-1)^2+4}} \\ \frac{1}{\sqrt{2}} & \frac{|2\epsilon-1|}{\sqrt{2(2\epsilon-1)^2+4}} \end{pmatrix} \Bigg), 
\endgroup  
\ \
\mathcal{T}(\epsilon)=\{0\}, \ \
\mathcal{N}(\epsilon)=\mathcal{R}\Bigg(\begin{pmatrix} \frac{(1-2\epsilon)}{\sqrt{(2\epsilon-1)^2+2}}\\\frac{-1}{\sqrt{(2\epsilon-1)^2+2}}\\\frac{1}{\sqrt{(2\epsilon-1)^2+2}} \end{pmatrix}\Bigg), 
\end{aligned}
\end{equation}
in which $\mathrm{sgn}(.)$ denotes the signum function. We can observe that, while being invariant w.r.t. $\epsilon$ in $(-\infty,-\frac12)$ and $(\frac32,\infty)$, the optimal partition changes with $\epsilon$ in $(-\frac12,\frac32)$, even though both $\rank\!\big(X^*(\epsilon)\big)$ and $\rank\!\big(S^*(\epsilon)\big)$ remain constant. A strictly complementary solution $\big(X^*(\epsilon),y^*(\epsilon),S^*(\epsilon)\big)$ exists for all $\epsilon \in (-\frac12,\frac32)$, and it is given by
\begin{align*}
X^*(\epsilon)&=
\begingroup 
\setlength\arraycolsep{.5pt}
\begin{pmatrix} 1 & \ \frac12 - \epsilon & \ \epsilon - \frac12 \\ \frac12-\epsilon & \ 1 & \ 1-2(\epsilon-\frac12)^2\\ \epsilon-\frac12 & \ 1-2(\epsilon-\frac12)^2 & \ 1 \end{pmatrix} 
\endgroup, \quad
S^*(\epsilon)=
\begingroup 
\setlength\arraycolsep{.5pt}
\begin{pmatrix} (2\epsilon-1)^2 & \ 2\epsilon-1 & \ 1- 2\epsilon \\ 2\epsilon-1 & \ \ 1 & -1\\ 1- 2\epsilon & -1 & \ \ 1 \end{pmatrix}
\endgroup\\
y^*(\epsilon)&=(-(2\epsilon-1)^2, \ -1, \ -1)^T,  
\end{align*}
\end{example}

\vspace{5px}
\noindent
As indicated in~\cite{GS99} and also demonstrated by Example~\ref{motivation_nonlinearity}, the optimal partition may vary with $\epsilon$ on a subinterval of $\interior(\mathcal{E})$. However, the dimensions of $\mathcal{B}(\epsilon)$ and $\mathcal{N}(\epsilon)$, or equivalently $\rank\!\big(X^*(\epsilon)\big)$ and $\rank\!\big(S^*(\epsilon)\big)$, might be stable on certain subintervals. This is in contrast to LO and LCQO, where the interval $\mathcal{E}$ is divided into subintervals and transition points each with a unique optimal partition.

\vspace{5px}
\noindent
Motivated by this observation, we review the notion of an invariancy set from~\cite{GS99} and then introduce nonlinearity intervals and transition points of the optimal partition for $(\mathrm{P_{\epsilon}})$ and $(\mathrm{D_{\epsilon}})$.
%%%%%%%%
%New Section
%%%%%%%%
\subsection{Invariancy intervals}
Let $\mathcal{I}_{\mathrm{inv}}$ be a subset of $\interior(\mathcal{E})$. Then $\mathcal{I}_{\mathrm{inv}}$ is called an invariancy set if $\pi(\epsilon')=\pi(\epsilon'')$ for all $\epsilon',\epsilon'' \in \mathcal{I}_{\mathrm{inv}}$. The following result is an extension from LCQO~\cite{BJRT96}.
%
 %%%%%%%
%New Lemma
 %%%%%%%
\begin{lemma}\label{constancy_interval}
Let $\big(X^*(\epsilon'),y^*(\epsilon'),S^*(\epsilon')\big)$ and $\big(X^*(\epsilon''),y^*(\epsilon''),S^*(\epsilon'')\big)$ be maximally complementary solutions, where $\epsilon',\epsilon'' \in \interior(\mathcal{E})$. If $\pi(\epsilon') = \pi(\epsilon'')$ and $\epsilon_{\rho}:=\rho \epsilon' + (1-\rho)\epsilon''$ for every $0 \le \rho \le 1$, then $\pi(\epsilon') = \pi(\epsilon'') = \pi(\epsilon_{\rho})$. Moreover, 
\begin{equation}\label{maximally_complementary_convex}
\begin{aligned}
\tilde{X}(\epsilon_{\rho}) &:= \rho X^*(\epsilon') + (1-\rho)X^*(\epsilon''),\\
\tilde{y}(\epsilon_{\rho}) &:= \rho y^*(\epsilon') + (1-\rho)y^*(\epsilon''),\\
\tilde{S}(\epsilon_{\rho}) &:= \rho S^*(\epsilon') + (1-\rho)S^*(\epsilon'')
\end{aligned}
\end{equation}
is a maximally complementary solution of $(\mathrm{P_{\epsilon_{\rho}}})$ and $(\mathrm{D_{\epsilon_{\rho}}})$.
\end{lemma}
\begin{proof}
Since $\mathcal{B}(\epsilon') = \mathcal{B}(\epsilon'')$ and $\mathcal{N}(\epsilon') = \mathcal{N}(\epsilon'')$, it is easy to see from the primal-dual feasibility constraints that $\big(\tilde{X}(\epsilon_{\rho}),\tilde{y}(\epsilon_{\rho}),\tilde{S}(\epsilon_{\rho})\big)$ is a primal-dual feasible solution of $(\mathrm{P_{\epsilon_{\rho}}})$ and $(\mathrm{D_{\epsilon_{\rho}}})$. Let us fix $Q_{\mathcal{B}(\epsilon')}$ and $Q_{\mathcal{N}(\epsilon')}$. Then, by Theorem~\ref{optimal_sol_rep}, there exist $U_{X^*(\epsilon'')} \succ 0$ and $U_{S^*(\epsilon'')} \succ 0$ such that 
\begin{align*}
\tilde{X}(\epsilon_{\rho}) &= \rho Q_{\mathcal{B}(\epsilon')}  U_{X^*(\epsilon')} Q^T_{\mathcal{B}(\epsilon')} + (1-\rho) Q_{\mathcal{B}(\epsilon')} U_{X^*(\epsilon'')} Q^T_{\mathcal{B}(\epsilon')},\\
\tilde{S}(\epsilon_{\rho}) &= \rho Q_{\mathcal{N}(\epsilon')} U_{S^*(\epsilon')} Q^T_{\mathcal{N}(\epsilon')} + (1-\rho) Q_{\mathcal{N}(\epsilon')} U_{S^*(\epsilon'')} Q^T_{\mathcal{N}(\epsilon')},
\end{align*}
and 
\begin{align*}
\rho U_{X^*(\epsilon')} + (1-\rho) U_{X^*(\epsilon'')} &\succ 0,\\
 \rho U_{S^*(\epsilon')} + (1-\rho) U_{S^*(\epsilon'')} &\succ 0.
\end{align*}
All this implies that $\big(\tilde{X}(\epsilon_{\rho}),\tilde{y}(\epsilon_{\rho}),\tilde{S}(\epsilon_{\rho})\big)$ is a primal-dual optimal solution of $(\mathrm{P_{\epsilon_{\rho}}})$ and $(\mathrm{D_{\epsilon_{\rho}}})$ and
\begin{equation}\label{inclusion_property}
\begin{aligned}
\mathcal{B}(\epsilon') &= \mathcal{R}\big(Q_{\mathcal{B}(\epsilon')}\big)= \mathcal{R}\big(\tilde{X}(\epsilon_{\rho})\big) \subseteq \mathcal{B}(\epsilon_{\rho}), \\
\mathcal{N}(\epsilon') &= \mathcal{R}\big(Q_{\mathcal{N}(\epsilon')}\big) =  \mathcal{R}\big(\tilde{S}(\epsilon_{\rho})\big) \subseteq \mathcal{N}(\epsilon_{\rho}),
\end{aligned}  
\end{equation}
where the inclusions follow from the definition of a maximally complementary solution. Using the same argument and Theorem~\ref{optimal_sol_rep}, we can choose a sufficiently small $\kappa$ so that
\begin{equation}\label{intermediate_sol}
\begin{aligned}
\hat{X}\big((1 + \kappa)\epsilon'' - \kappa \epsilon'\big) &:= (1+\kappa) X^*(\epsilon'') - \kappa  X^*(\epsilon'),\\
\hat{y}\big((1 + \kappa)\epsilon'' - \kappa \epsilon'\big) &:= (1+\kappa) y^*(\epsilon'') - \kappa  y^*(\epsilon'),\\
\hat{S}\big((1 + \kappa)\epsilon'' - \kappa \epsilon'\big) &:=(1+\kappa) S^*(\epsilon'') - \kappa  S^*(\epsilon') 
\end{aligned}
\end{equation}
yields an optimal solution for $(\mathrm{P_{(1 + \kappa)\epsilon'' - \kappa \epsilon'}})$ and $(\mathrm{D_{(1 + \kappa)\epsilon'' - \kappa \epsilon'}})$. Note that $\kappa$ can be made so small that $(1 + \kappa)\epsilon'' - \kappa \epsilon' \in \interior(\mathcal{E})$. Now, if $\mathcal{T}(\epsilon') \supsetneq \mathcal{T}(\epsilon_{\rho})$, then there would exist a maximally complementary solution $\big(X^*(\epsilon_{\rho}),y^*(\epsilon_{\rho}),S^*(\epsilon_{\rho})\big)$ and $0 \neq q \in \mathcal{T}(\epsilon')$ so that 
\begin{align}\label{common_eigenvector}
q^T \big(X^*(\epsilon_{\rho}) + S^*(\epsilon_{\rho})\big)q > 0.
\end{align}
However, this would contradict the optimal partition at $\epsilon'$ and $\epsilon''$. To see this, we can check that 
\begin{align*}
\epsilon'' = \frac{\kappa}{\kappa + \rho} \epsilon_{\rho} + \frac{\rho}{\kappa + \rho} \epsilon''',
\end{align*}
where $ \epsilon''':= (1 + \kappa)\epsilon'' - \kappa \epsilon'$. Then 
\begin{align*}
\bar{X}(\epsilon'') &:=  \frac{\kappa}{\kappa + \rho} X^*(\epsilon_{\rho}) + \frac{\rho}{\kappa + \rho}\hat{X}(\epsilon'''),\\
\bar{y}(\epsilon'') &:=  \frac{\kappa}{\kappa + \rho} y^*(\epsilon_{\rho}) + \frac{\rho}{\kappa + \rho}\hat{y}(\epsilon'''),\\
\bar{S}(\epsilon'') &:=  \frac{\kappa}{\kappa + \rho} S^*(\epsilon_{\rho}) + \frac{\rho}{\kappa + \rho}\hat{S}(\epsilon''')
\end{align*}
gives a primal-dual optimal solution for $(\mathrm{P_{\epsilon''}})$ and $(\mathrm{D_{\epsilon''}})$, where $\big(\bar{X}(\epsilon''),\bar{y}(\epsilon''),\bar{S}(\epsilon'')\big)$ satisfies the complementarity condition by~\eqref{inclusion_property} and~\eqref{intermediate_sol}. However, we have from~\eqref{common_eigenvector} that
\begin{align*}
q^T \big(\bar{X}(\epsilon'') + \bar{S}(\epsilon'')\big)q > 0,
\end{align*}
which is a contradiction, since $\bar{X}(\epsilon'')q=\bar{S}(\epsilon'')q=0$ by Theorem~\ref{optimal_sol_rep}. Therefore, we have $\mathcal{T}(\epsilon') = \mathcal{T}(\epsilon_{\rho})$, which induces $\mathcal{B}(\epsilon') = \mathcal{B}(\epsilon_{\rho})$ and $\mathcal{N}(\epsilon') = \mathcal{N}(\epsilon_{\rho})$. The second part of the proof is immediate.
\end{proof}
\noindent
Let $\bar{\epsilon} \in \mathcal{I}_{\mathrm{inv}}$. By the definition of an invariancy set, $\mathcal{I}_{\mathrm{inv}}$ is the set of all $\epsilon \in \interior(\mathcal{E})$ for which the system
\begin{align*}
\langle A^i , Q_{\mathcal{B}(\bar{\epsilon} )} U_X Q^T_{\mathcal{B}(\bar{\epsilon} )} \rangle&=b_i, \qquad   i=1,\ldots, m, & U_X &\succ 0,\\[-1\jot]
\sum_{i=1}^m A^i y_i+Q_{\mathcal{N}(\bar{\epsilon} )} U_S Q^T_{\mathcal{N}(\bar{\epsilon} )}&=C + \epsilon \bar{C}, & U_S &\succ 0
\end{align*}
remains feasible. Therefore, from Lemma~\ref{constancy_interval} it is immediate that $\mathcal{I}_{\mathrm{inv}}$ is either a singleton or an open, possibly unbounded, interval. The latter is simply referred to as an invariancy interval.
%%%%%%%%
%New Remark
%%%%%%%% 
\begin{remark}
It follows from~\eqref{maximally_complementary_convex} that $v(\epsilon_{\rho}) =  \rho v(\epsilon') + (1-\rho) v(\epsilon'')$, i.e., the optimal value function is indeed linear on an invariancy set. Furthermore, it is easy to show that there exists either a unique primal optimal solution or a unique primal optimal set associated with an invariancy set, see also Corollary 2 in~\cite{JRT93}. For instance, the invariancy interval $(-\infty,-\frac12)$ in Example~\ref{motivation_nonlinearity} corresponds to the unique primal optimal solution
\begin{align*}
X^*(\epsilon)&=\begin{pmatrix} \ \ 1 & \ \ \ 1 & \ -1 \\ \ \ 1 & \ \ \ 1 & \ -1\\ -1 & \ -1 & \ \ \ 1 \end{pmatrix}, \qquad \epsilon \in (-\infty,-\frac12),
\end{align*}
which is an extreme point of $\mathcal{E}\ell\ell_3$.
\end{remark}
\vspace{5px}
\noindent
An invariancy set can be computed by solving a pair of auxiliary SDO problems. The linear conic optimization counterpart can be found in Section 4 in~\cite{Y2004}.
%
 %%%%%%%
%New Lemma
 %%%%%%%
\begin{lemma}[Lemma 4.1 in~\cite{GS99}]\label{linearity_interval}
Assume that $\bar{\epsilon}$ belongs to a bounded invariancy set $\mathcal{I}_{\mathrm{inv}}$. Then the boundary points of $\mathcal{I}_{\mathrm{inv}}$ can be obtained by solving 
\begin{align*}
\alpha_{\mathrm{inv}}(\beta_{\mathrm{inv}}):=\inf(\sup)  &\quad \epsilon\\[-1\jot] 
\st &\quad \sum_{i=1}^m y_i A^i + Q_{\mathcal{N}(\bar{\epsilon})} U_S Q^T_{\mathcal{N}(\bar{\epsilon})} = C + \epsilon \bar{C},\\
&\quad U_S \succ 0.
\end{align*}
If $\mathcal{I}_{\mathrm{inv}}$ is unbounded, then we have either $\alpha_{\mathrm{inv}} = -\infty$, $\beta_{\mathrm{inv}} = \infty$, or both. \qed
\end{lemma} 
%%%%%%%%
%New Section
%%%%%%%%
\subsection{Transition points and nonlinearity intervals}
As a result of Lemma~\ref{linearity_interval}, if $\alpha_{\mathrm{inv}} < \beta_{\mathrm{inv}}$, then $\alpha_{\mathrm{inv}} < \bar{\epsilon} < \beta_{\mathrm{inv}}$ belongs to the invariancy interval $(\alpha_{\mathrm{inv}},\beta_{\mathrm{inv}})$. Otherwise, $\alpha_{\mathrm{inv}}=\bar{\epsilon} = \beta_{\mathrm{inv}}$ indicates that the optimal partition changes in every neighborhood of $\bar{\epsilon}$. In Example~\ref{motivation_nonlinearity}, $(-\frac12,\frac32)$ is a subinterval with varying optimal partition.  

\begin{definition}\label{transition_point}
A singleton invariancy set $\{\bar{\epsilon}\} \in \interior(\mathcal{E})$ is called a transition point if for every $\xi > 0$ there exists $\epsilon \in (\bar{\epsilon}-\xi,\bar{\epsilon}+\xi) \subseteq \interior(\mathcal{E})$ such that  
\begin{align*}
\dim\!\big(\mathcal{B}(\epsilon)\big) \neq \dim\!\big(\mathcal{B}(\bar{\epsilon})\big), \quad \text{or} \quad \dim\!\big(\mathcal{N}(\epsilon)\big) \neq \dim\!\big(\mathcal{N}(\bar{\epsilon})\big). 
\end{align*}
\end{definition}
\begin{definition}\label{nonlinearity_interval}
A nonlinearity interval is defined as a non-singleton open, possibly unbounded, subinterval of maximal length $\mathcal{I}_{\mathrm{non}} \subseteq \interior(\mathcal{E})$ such that
\begin{align*}
\dim\!\big(\mathcal{B}(\epsilon')\big)=\dim\!\big(\mathcal{B}(\epsilon'')\big), \quad \text{and} \quad \dim\!\big(\mathcal{N}(\epsilon')\big)=\dim\!\big(\mathcal{N}(\epsilon'')\big), \qquad \forall \epsilon',\epsilon'' \in \mathcal{I}_{\mathrm{non}},
\end{align*}
while $\pi(\epsilon)$ varies with $\epsilon$.
\end{definition} 

\vspace{5px}
\noindent
Given a maximally complementary solution $\big(X^*(\epsilon),y^*(\epsilon),S^*(\epsilon)\big)$, Definition~\ref{nonlinearity_interval} yields the fact that the eigenvalues of both $X^*(\epsilon)$ and $S^*(\epsilon)$ change with $\epsilon$ on a nonlinearity interval. Furthermore, Definition~\ref{transition_point} implies that a singleton invariancy set $\{\bar{\epsilon}\}$ either is a transition point, or it lies in a nonlinearity interval. 
%%%%%%%%
%New Section
%%%%%%%%
\subsubsection{On the existence of a nonlinearity interval}
Observe from Example~\ref{motivation_nonlinearity} that $\big(X^*(\epsilon),y^*(\epsilon),S^*(\epsilon)\big)$ is strictly complementary, and the eigenvalues of $X^*(\epsilon)$ and $S^*(\epsilon)$ are continuous on $(-\frac12,\frac32)$:
\begin{equation}
\begin{aligned}\label{eigenvalues_example}
\Lambda\big(X^*(\epsilon)\big)=\begin{pmatrix} -2\epsilon^2+2\epsilon+\frac32 & 0 & 0\\0 & 2\epsilon^2-2\epsilon+\frac32 & 0\\0 & 0 & 0 \end{pmatrix}, \quad
\Lambda\big(S^*(\epsilon)\big)=\begin{pmatrix} 0 & 0 & 0\\0 & 0 & 0\\0 & 0 & 4\epsilon^2-4\epsilon+3 \end{pmatrix}.
\end{aligned}
\end{equation}

\vspace{5px}
\noindent
Mathematically speaking, continuity arguments and the strict complementarity condition induce sufficient conditions for the existence of a nonlinearity interval, as stated in Theorem~\ref{nonlinearity_continuity}.
 %%%%%%%%
%New Theorem
 %%%%%%%%
\begin{theorem}\label{nonlinearity_continuity}
Assume that for every sequence $\{\epsilon_k\} \to \bar{\epsilon}$ there exists a sequence of optimal solutions $\big\{\big(X(\epsilon_k),y(\epsilon_k),S(\epsilon_k)\big)\big\} \to \big(X^*(\bar{\epsilon}),y^*(\bar{\epsilon}),S^*(\bar{\epsilon})\big)$, where $ \big(X^*(\bar{\epsilon}),y^*(\bar{\epsilon}),S^*(\bar{\epsilon})\big)$ is a strictly complementary solution. Then $\bar{\epsilon}$ belongs to a nonlinearity interval.
\end{theorem}
\begin{proof}
Since $\big(X^*(\bar{\epsilon}),y^*(\bar{\epsilon}),S^*(\bar{\epsilon})\big)$ is strictly complementary, we have
\begin{align}
\rank\!\big(X^*(\bar{\epsilon})\big) + \rank\!\big(S^*(\bar{\epsilon})\big) = n. \label{strict_complementarity_con}
\end{align}
Then by the assumptions and the continuity of the eigenvalues, there exists a primal-dual optimal solution $\big(X(\epsilon_k),y(\epsilon_k),S(\epsilon_k)\big)$ such that
\begin{align*}
\rank\!\big(X(\epsilon_k)\big) &\ge \rank\!\big(X^*(\bar{\epsilon})\big),\\
\rank\!\big(S(\epsilon_k)\big) &\ge \rank\!\big(S^*(\bar{\epsilon})\big)
\end{align*}
for sufficiently large $k$, which by~\eqref{strict_complementarity_con} imply that $\big(X(\epsilon_k),y(\epsilon_k),S(\epsilon_k)\big)$ is strictly complementary.
\end{proof} 
%%%%%%%%
%New Remark
%%%%%%%%
\begin{remark}\label{continuity_singlevalued}
Note that $\mathcal{P}^*(\epsilon)$ and $\mathcal{D}^*(\epsilon)$ are uniformly bounded near any $\epsilon \in  \interior(\mathcal{E})$, see e.g., Lemma 3.11 in~\cite{SW16}. Hence, in the special case when both $\mathcal{P}^*(\bar{\epsilon})$ and $\mathcal{D}^*(\bar{\epsilon})$ are singleton, the continuity condition of Theorem~\ref{nonlinearity_continuity} automatically holds, see e.g., Corollary 8.1 in~\cite{Ho73b}.
\end{remark}

\vspace{5px}
\noindent
In Example~\ref{motivation_nonlinearity}, we can compute the boundary points of the nonlinearity interval using the explicit form of the eigenvalues given in~\eqref{eigenvalues_example}. In practice, however, it may not be possible to obtain explicit formulas for the eigenvalues in terms of $\epsilon$. More importantly, even with the existence of the strict complementarity condition, the continuity of the eigenvalues of a strictly complementary solution may not be possible to verify in practice. Therefore, identification of a nonlinearity interval is, in general, a nontrivial task.

\vspace{5px}
\noindent
Recall from Definitions~\ref{transition_point} and~\ref{nonlinearity_interval} that a transition point $\bar{\epsilon}$ lies on the boundary of an invariancy or a nonlinearity interval. Equivalently, every neighborhood of $\bar{\epsilon}$ contains an $\epsilon$ with $\big(\mathcal{B}(\epsilon),\mathcal{T}(\epsilon),\mathcal{N}(\epsilon)\big)$ having different dimensions from $\big(\mathcal{B}(\bar{\epsilon}),\mathcal{T}(\bar{\epsilon}),\mathcal{N}(\bar{\epsilon})\big)$. If $\bar{\epsilon}$ is adjacent to an invariancy interval $\mathcal{I}_{\mathrm{inv}}$, then this is consistently true for every $\xi > 0$ and every $\epsilon \in (\bar{\epsilon}-\xi,\bar{\epsilon}+\xi) \cap \mathcal{I}_{\mathrm{inv}}$. For an $\bar{\epsilon}$ adjacent to a nonlinearity interval, unless additional local information is provided, one may not conclude the same property. More precisely, it is not immediate, solely from Definition~\ref{nonlinearity_interval}, whether 
\begin{equation}
\dim\!\big(\mathcal{B}(\epsilon)\big) \neq \dim\!\big(\mathcal{B}(\bar{\epsilon})\big), \  \text{or} \  \dim\!\big(\mathcal{N}(\epsilon)\big) \neq \dim\!\big(\mathcal{N}(\bar{\epsilon})\big), \forall \epsilon \in (\bar{\epsilon}-\xi,\bar{\epsilon}+\xi) \cap \mathcal{I}_{\mathrm{non}}. \label{difference_transition_partititon}
\end{equation}
Corollary~\ref{openness} spells out sufficient conditions to guarantee the change of rank at a boundary point of a nonlinearity interval.
 %%%%%%%
%New Lemma
 %%%%%%%
\begin{corollary}\label{openness}
Let $\mathcal{I}_{\mathrm{non}}$ be a nonlinearity interval satisfying the strict complementarity condition, and let $\bar{\epsilon}$ be a boundary point of $\mathcal{I}_{\mathrm{non}}$. If for every $\{\epsilon_k\} \to \bar{\epsilon}$ there exists a sequence of optimal solutions $\big\{\big(X(\epsilon_k),y(\epsilon_k),S(\epsilon_k)\big)\big\}$ converging to a maximally complementary solution $\big(X^*(\bar{\epsilon}),y^*(\bar{\epsilon}),S^*(\bar{\epsilon})\big)$, then~\eqref{difference_transition_partititon} holds.

\end{corollary}
\begin{proof}
Since $\bar{\epsilon}$ is a transition point and the eigenvalues of $X^*(\epsilon)$ and $S^*(\epsilon)$ vary continuously in a small neighborhood of $\bar{\epsilon}$, the strict complementarity condition must fail at $\bar{\epsilon}$. Otherwise, $\bar{\epsilon}$ would belong to a nonlinearity interval by Theorem~\ref{nonlinearity_continuity}, which is a contradiction.
\end{proof}

 %%%%%%%%
%New Remark
 %%%%%%%%
\begin{remark}
If the continuity condition in Corollary~\ref{openness} fails, then either $\mathcal{P}^*(\bar{\epsilon})$ or $\mathcal{D}^*(\bar{\epsilon})$ must not be singleton, see Remark~\ref{continuity_singlevalued}. In this case, the sequence $\big\{\big(X(\epsilon_k),y(\epsilon_k),S(\epsilon_k)\big)\big\}$ converges to the boundary of the optimal set at $\bar{\epsilon}$, and thus it may provide no useful information about the rank of $X^*(\bar{\epsilon})$ or $S^*(\bar{\epsilon})$.   
\end{remark}

%%%%%%%%
%New Section
%%%%%%%%
\subsubsection{On the existence of a transition point}
Analogous to a nonlinearity interval, in general, it is not trivial to identify a transition point of the optimal partition. This is in contrast to LO and LCQO cases, where transition points and non-differentiable points of the optimal value function coincide, see e.g., Theorem 3.7 in~\cite{BJRT96}. For instance, by appending the redundant inequality constraint $z \le 1$ to Example~\ref{motivation_nonlinearity}, we get a new parametric SDO problem 
\begin{align}
\min\!\Bigg\{ (4\epsilon-2)x+(2-4\epsilon)y-2z \ \Bigg | \begin{pmatrix} 1 & \ x & \ y & 0\\x & \ 1 & \ z & 0\\y & \ z & \ 1 & 0\\0 & 0 & 0 & 1-z \end{pmatrix} \succeq 0, \quad (x,y,z)\in\mathbb{R}^3\Bigg\}. \label{3elliptope_appended}
\end{align}
The optimal partition of~\eqref{3elliptope_appended} has a transition point at $\epsilon = \frac12$, while the optimal value function is analytic on $(-\frac12,\frac32)$.

\vspace{5px}
\noindent
A transition point can be further characterized using nonsingularity of the Jacobian of the optimality conditions. Note that the optimality conditions for $(\mathrm{P_{\epsilon}})$ and $(\mathrm{D_{\epsilon}})$ can be written as 
\begin{equation}\label{KKT_conditions}
\begin{aligned}
\mathcal{A}\svectorize(X) &= b,\\[-1\jot]
\mathcal{A}^Ty + \svectorize(S) &= \svectorize(C) + \epsilon \svectorize(\bar{C}),\\[-1\jot]
\frac12 \svectorize(XS + SX) &= 0,\\[-1\jot]
X,S &\succeq 0.
\end{aligned}
\end{equation}
Then the Jacobian of the linear equations in~\eqref{KKT_conditions} is given by
\begin{align*}
J(X,y,S):=\begin{pmatrix} \mathcal{A} & 0 & 0\\0 & \mathcal{A}^T & I_{n(n+1)/2} \\ S \otimes_s I_n & 0 & X \otimes_s I_n \end{pmatrix},
\end{align*}
where $\mathcal{A}$ and the linear transformation $\svectorize(.)$ are defined in~\eqref{coefficient_transformation} and~\eqref{linear_transformation_svec}, respectively, and $\otimes_s$ denotes the \textit{symmetric Kronecker product}%
\footnote{The symmetric Kronecker product of any two square matrices $K_1$ and $K_2$ is defined as a mapping 
\begin{align*}
(K_1 \otimes_s K_2)\svectorize(H):= \frac12 \svectorize\big(K_2 H K_1^T + K_1HK_2^T\big),
\end{align*}
where $H$ is \textcolor{blue}{any} symmetric matrix. See e.g.,~\cite{Kl02} for more details.}. The following technical lemma is in order. 
%
 %%%%%%%
%New Lemma
 %%%%%%%
\begin{lemma}[Theorem 3.1 in~\cite{AHO98} and~\cite{Ha}]\label{nonsingularity_conditions}
Let $\big(X^*(\bar{\epsilon}),y^*(\bar{\epsilon}),S^*(\bar{\epsilon})\big)$ be a maximally complementary solution. Then $J(X^*(\bar{\epsilon}),y^*(\bar{\epsilon}),S^*(\bar{\epsilon})\big)$ is nonsingular if and only if $\big(X^*(\bar{\epsilon}),y^*(\bar{\epsilon}),S^*(\bar{\epsilon})\big)$ is strictly complementary and both primal and dual nondegenerate.  \qed
\end{lemma}

\vspace{5px}
\noindent
Consequently, it can be deducted from Lemma~\ref{nonsingularity_conditions} and the implicit function theorem~\cite{RD09} that if $\big(X^*(\bar{\epsilon}),y^*(\bar{\epsilon}),S^*(\bar{\epsilon})\big)$ is unique and strictly complementary, then there exists $\xi > 0$ so that $\big(X^*(\epsilon),y^*(\epsilon),S^*(\epsilon)\big)$ is unique and continuously differentiable on $(\bar{\epsilon}-\xi,\bar{\epsilon}+\xi)$. This together with Theorem~\ref{nonlinearity_continuity} implies the existence of a nonlinearity interval around $\bar{\epsilon}$. Under additional conditions, spelled out in Theorem~\ref{nonsingularity_sufficient}, this nonlinearity interval coincides with the open interval on which the Jacobian is nonsingular.
 %%%%%%%%
%New Theorem
 %%%%%%%%
\begin{theorem}\label{nonsingularity_sufficient}
Let $\mathcal{I}_{\mathrm{reg}}$ be an open interval of maximal length on which $J(X^*(\epsilon),y^*(\epsilon),S^*(\epsilon)\big)$ is nonsingular. If the strict complementarity condition fails at a boundary point of $\mathcal{I}_{\mathrm{reg}}$, if there exists any, then the boundary point is a transition point of the optimal partition. In particular, the result holds when both $\mathcal{P}^*(\epsilon)$ and $\mathcal{D}^*(\epsilon)$ are singleton at the boundary point of $\mathcal{I}_{\mathrm{reg}}$.
\end{theorem}
\begin{proof}
The first part is immediate, since at least one of $\rank\big(X^*(\epsilon)\big)$ or $\rank\big(S^*(\epsilon)\big)$ must decrease at the boundary point of $\mathcal{I}_{\mathrm{reg}}$. The second part implies that the strict complementarity condition must fail at the boundary point, since otherwise the Jacobian would be nonsingular.
\end{proof}

\vspace{5px}
\noindent
Observe that if the strict complementarity condition holds at a boundary point of $\mathcal{I}_{\mathrm{reg}}$, then $\mathcal{I}_{\mathrm{reg}}$ might be just a subinterval of a nonlinearity interval. This case can be demonstrated by Example~\ref{motivation_nonlinearity} where the Jacobian is singular at a non-transition point $\epsilon=\frac12$. To see the nonsingularity elsewhere, using a common orthonormal eigenvector basis~\eqref{optimal_partition_nonlinearity} and the conditions in Section~\ref{Prima_dual_nondeg}, one can check that the matrices
\begin{align*}
\begingroup 
\setlength\arraycolsep{.5pt}
\begin{pmatrix} 0 & \ 0 & \ 0\\0 & \ \big(\frac{2\mathrm{sgn}(2\epsilon-1)}{\nu_1}\big)^2 & \ \frac{2(1-2\epsilon)\mathrm{sgn}(2\epsilon-1)}{\nu_1\nu_2}\\0 & \ \frac{2(1-2\epsilon)\mathrm{sgn}(2\epsilon-1)}{\nu_1\nu_2} & \ 0 \end{pmatrix} \endgroup, 
\begingroup 
\setlength\arraycolsep{.5pt}
\begin{pmatrix} \frac12 & \ -\frac{|2\epsilon-1|}{\sqrt{2}\nu_1} & \ -\frac{1}{\sqrt{2}\nu_2}\\ -\frac{|2\epsilon-1|}{\sqrt{2}\nu_1} & \ \big(\frac{2\epsilon-1}{\nu_1}\big)^2 & \ \frac{|2\epsilon-1|}{\nu_1\nu_2}\\-\frac{1}{\sqrt{2}\nu_2} & \ \frac{|2\epsilon-1|}{\nu_1\nu_2} & \ 0 \end{pmatrix} \endgroup, 
\begingroup 
\setlength\arraycolsep{.5pt}
\begin{pmatrix} \frac12 & \ \frac{|2\epsilon-1|}{\sqrt{2}\nu_1} & \ \frac{1}{\sqrt{2}\nu_2}\\ \frac{|2\epsilon-1|}{\sqrt{2}\nu_1} & \ \big(\frac{2\epsilon-1}{\nu_1}\big)^2 & \ \frac{|2\epsilon-1|}{\nu_1\nu_2}\\ \frac{1}{\sqrt{2}\nu_2} & \ \frac{|2\epsilon-1|}{\nu_1\nu_2} & \ 0 \end{pmatrix} \endgroup
\end{align*}
are linearly independent for all $\epsilon \in (-\frac12,\frac32) \setminus \{\frac12\}$, where
\begin{align*}
\nu_1:=\sqrt{2(2\epsilon-1)^2+4}, \quad \text{and} \quad \nu_2:=\sqrt{(2\epsilon-1)^2+2}.
\end{align*}
Furthermore, we can observe that the following matrices span $\mathbb{S}^2$:
\begin{align*}
\begingroup 
\setlength\arraycolsep{.5pt}
\begin{pmatrix} 0 & \ 0 \\0 & \ \big(\frac{2\mathrm{sgn}(2\epsilon-1)}{\nu_1}\big)^2 \end{pmatrix} \endgroup,  
\begingroup 
\setlength\arraycolsep{.5pt}\begin{pmatrix} \frac12 & \ -\frac{|2\epsilon-1|}{\sqrt{2}\nu_1} \\-\frac{|2\epsilon-1|}{\sqrt{2}\nu_1} & \ \big(\frac{2\epsilon-1}{\nu_1}\big)^2 \end{pmatrix} \endgroup,   
\begingroup 
\setlength\arraycolsep{.5pt}
\begin{pmatrix} \frac12 & \ \frac{|2\epsilon-1|}{\sqrt{2}\nu_1} \\ \frac{|2\epsilon-1|}{\sqrt{2}\nu_1} & \ \big(\frac{2\epsilon-1}{\nu_1}\big)^2 \end{pmatrix}\endgroup,
\end{align*}
which implies the nondegeneracy of the unique dual optimal solution for all $\epsilon \in (-\frac12,\frac32) \setminus \{\frac12\}$. At $\epsilon=\frac12$ the dual nondegeneracy condition does not hold, since the matrices
\begin{align*}
\begin{pmatrix} 0 & 1\\\frac{1}{\sqrt{2}} & 0\\ \frac{1}{\sqrt{2}} & 0 \end{pmatrix}^T A^i \begin{pmatrix} 0 & 1\\ \frac{1}{\sqrt{2}} & 0\\ \frac{1}{\sqrt{2}} & 0 \end{pmatrix}, \qquad i=1,\ldots,m
\end{align*}
fail to span $\mathbb{S}^2$. All this yields the nonsingularity of the Jacobian on $(-\frac12,\frac32)\setminus\{\frac12\}$. 

\vspace{5px}
\noindent
Theorem~\ref{nonlinearity_continuity} indicates that at a transition point $\bar{\epsilon}$ which satisfies the strict complementarity condition, the eigenvalues of $X^*(\epsilon)$ or $S^*(\epsilon)$ must be discontinuous. Thus, the following result is immediate.
\begin{corollary}\label{transition_point_property}
At a transition point $\bar{\epsilon}$, at least one of the strict complementarity, primal nondegeneracy, or dual nondegeneracy conditions has to fail.  
\end{corollary}

\begin{proof}
If all the conditions hold, then $\bar{\epsilon}$ would belong to a nonlinearity interval by Lemma~\ref{nonsingularity_conditions} and the subsequent discussion.   
\end{proof}
\noindent
In other words, the Jacobian of the optimality conditions must be singular at a transition point. However, the reverse direction is not true as can be verified in Example~\ref{motivation_nonlinearity}. For this case, the dual nondegeneracy condition fails at $\epsilon=\frac12$, while $\epsilon$ is not a transition point.
%
%%%%%%%%
%New Section
%%%%%%%%
\section{Sensitivity of the approximation of the optimal partition}\label{perturbation_optimal_partition}
Thus far, we have investigated the sensitivity of the optimal set mapping at a transition point or on a nonlinearity interval. Given the fact that $\pi(\epsilon)$ varies on a nonlinearity interval, see~\eqref{optimal_partition_nonlinearity}, we would like to derive upper bounds on a metric which measures the sensitivity of the approximation of the optimal partition, i.e., the subspaces spanned by the eigenvectors whose accumulation points form orthonormal bases for the optimal partition. Throughout this section, unless stated otherwise, we always assume that $\mu$ is positive, and $\epsilon = 0$ belongs to a nonlinearity interval. For the sake of brevity, we drop $\epsilon$ from the central solution, optimal partition, and optimal solutions at $\epsilon = 0$. 

\vspace{5px}
\noindent
Consider an equivalent form of the perturbed central path equations as follows 
\begin{equation}\label{alternative_CP}
F\big(X,y,S,\mu,\epsilon\big):=\begin{pmatrix}
\mathcal{A}\svectorize(X) - b\\
\mathcal{A}^Ty + \svectorize(S) - \svectorize(C) - \epsilon\svectorize(\bar{C})\\
\svectorize(XS + SX - 2\mu I_n)\\ 
\end{pmatrix}=0, \quad
X,S \succeq 0.
\end{equation}
It can be shown that system~\eqref{alternative_CP} is solvable for all $\epsilon$ in a neighborhood of $0$. This directly follows from the nonsingularity of the Jacobian, see e.g., Theorem 3.3 in~\cite{Kl02}, the implicit function theorem~\cite{RD09}, and continuity arguments. For every $\epsilon$ the unique solution of~\eqref{alternative_CP} is denoted by $\big(X^{\mu}(\epsilon),y^{\mu}(\epsilon),S^{\mu}(\epsilon)\big)$ and a common eigenvector basis is represented by $Q^{\mu}_{\epsilon}$. The analogue of $\tilde{\mu}$ at $\epsilon$ is denoted by $\tilde{\mu}(\epsilon)$. 

\vspace{5px}
\noindent
Suppose that for $\epsilon=0$ a central solution $\big(X^{\mu},y^{\mu},S^{\mu}\big)$ is given, where $\mu < \tilde{\mu}$ as defined in~\eqref{upper_bound_partition}. The eigenvectors of $X^{\mu}$ and $S^{\mu}$ can be rearranged so that
\begin{align*}
Q^{\mu}:=\big(Q^{\mu}_{\mathcal{B}}, Q^{\mu}_{\mathcal{T}}, Q^{\mu}_{\mathcal{N}}\big).
\end{align*}
We quantify the sensitivity of $\mathcal{R}\big(Q^{\mu}_{\mathcal{B}}\big)$ and $\mathcal{R}\big(Q^{\mu}_{\mathcal{N}}\big)$, when $\epsilon$ belongs to a sufficiently small neighborhood of $0$ in the nonlinearity interval. We rely on the following theorem adopted from~\cite{GL13} and Theorem 4.11 in~\cite{St73}. For the ease of exposition, we have tailored the theorem for central solutions by introducing
\begin{align*}
\Xi_X^{\mu}(\epsilon) := X^{\mu}(\epsilon) - X^{\mu}, \qquad \text{and} \qquad \Xi_S^{\mu}(\epsilon) := S^{\mu}(\epsilon) - S^{\mu}.
\end{align*}
Recall that the distance between two subspaces is defined in~\eqref{metric}, which is a metric on the set of subspaces of $\mathbb{R}^n$~\cite{St73}.
%%%%%%%%
%New Theorem
%%%%%%%% 
\begin{theorem}\label{Stewart_eigenspace_theorem}
Let a central solution $\big(X^{\mu},y^{\mu},S^{\mu}\big)$ be given, and let $\epsilon$ belong to a nonlinearity interval in a neighborhood of $0$ such that 
\begin{align}
\mu &< \min\{\tilde{\mu},\tilde{\mu}(\epsilon)\}, \label{perturbed_threshold} \\
\big \|\Xi_X^{\mu}(\epsilon) \big \| &\le \frac{\lambda_{[n_{\mathcal{B}}]}(X^{\mu})-\lambda_{[n_{\mathcal{B}}+1]}(X^{\mu})}{5}, \label{condition_on_perturbation_X}\\
\big \|\Xi_S^{\mu}(\epsilon) \big \| &\le \frac{\lambda_{[n_{\mathcal{N}}]}(S^{\mu})-\lambda_{[n_{\mathcal{N}}+1]}(S^{\mu})}{5}  \nonumber
\end{align}
hold. Then there exist $V^{\mu}_{\epsilon} \in \mathbb{R}^{(n-n_{\mathcal{B}}) \times n_{\mathcal{B}}}$ and $W^{\mu}_{\epsilon} \in \mathbb{R}^{(n-n_{\mathcal{N}}) \times n_{\mathcal{N}}}$ such that the columns of $\big(Q^{\mu}_{\mathcal{B}} + Q^{\mu}_{\mathcal{T} \cup \mathcal{N}} V^{\mu}_{\epsilon}\big)\big(I_{n_{\mathcal{B}}}+(V^{\mu}_{\epsilon})^T V^{\mu}_{\epsilon}\big)^{-\frac12}$ and $\big(Q^{\mu}_{\mathcal{N}} + Q^{\mu}_{\mathcal{B} \cup \mathcal{T}} W^{\mu}_{\epsilon}\big)\big(I_{n_{\mathcal{N}}}+(W^{\mu}_{\epsilon})^TW^{\mu}_{\epsilon}\big)^{-\frac12}$ form orthonormal bases for $\mathcal{R}\big(Q_{\mathcal{B}(\epsilon)}^{\mu}\big)$ and $\mathcal{R}\big(Q_{\mathcal{N}(\epsilon)}^{\mu}\big)$. Furthermore, we have
\begin{align}
\distance\!\big(\mathcal{R}\big(Q^{\mu}_{\mathcal{B}}\big),\mathcal{R}\big(Q^{\mu}_{\mathcal{B}(\epsilon)}\big)\big) &\le \frac{4 \big \|(Q^{\mu}_{\mathcal{B}})^T \Xi_X^{\mu}(\epsilon) Q^{\mu}_{\mathcal{T} \cup \mathcal{N}} \big \|}{\lambda_{[n_{\mathcal{B}}]}(X^{\mu})-\lambda_{[n_{\mathcal{B}}+1]}(X^{\mu})}, \label{estimated_distance_B} \\
\distance\!\big(\mathcal{R}\big(Q^{\mu}_{\mathcal{N}}\big),\mathcal{R}\big(Q^{\mu}_{\mathcal{N}(\epsilon)}\big)\big) &\le \frac{4 \big \|(Q^{\mu}_{\mathcal{N}})^T \Xi_S^{\mu}(\epsilon) Q^{\mu}_{\mathcal{B} \cup \mathcal{T}} \big \|}{\lambda_{[n_{\mathcal{N}}]}(S^{\mu})-\lambda_{[n_{\mathcal{N}}+1]}(S^{\mu})}. \label{estimated_distance_N}
\end{align}
\end{theorem}

\begin{proof}
The proof is on the basis of perturbation bounds for invariant subspaces of a matrix, as stated in Theorem 8.1.10 in~\cite{GL13}. It is known that $\mathcal{R}\big(Q^{\mu}_{\mathcal{B}}\big)$, $\mathcal{R}\big(Q^{\mu}_{\mathcal{T}}\big)$, and $\mathcal{R}\big(Q^{\mu}_{\mathcal{N}}\big)$ are invariant subspaces of both $X^{\mu}$ and $S^{\mu}$, since, e.g., $X^{\mu} \mathcal{R}\big(Q^{\mu}_{\mathcal{B}}\big) \subseteq \mathcal{R}\big(Q^{\mu}_{\mathcal{B}}\big)$ and $S^{\mu} \mathcal{R}\big(Q^{\mu}_{\mathcal{N}}\big) \subseteq \mathcal{R}\big(Q^{\mu}_{\mathcal{N}}\big)$. We only state the proof for an invariant subspace of $X^{\mu}$.

\vspace{5px}
\noindent
Using the bounds in Lemma~\ref{bounds_on_muCenters} and~\eqref{perturbed_threshold}, it is easy to verify that 
\begin{align}
\lambda_{[n_{\mathcal{B}}]}(X^{\mu})-\lambda_{[n_{\mathcal{B}}+1]}(X^{\mu}) > 0. \label{lower_bound_sep_B}
\end{align}
All this implies that the eigenvalues of $(Q^{\mu}_{\mathcal{B}})^T X^{\mu}Q^{\mu}_{\mathcal{B}}$ are properly separated from the eigenvalues of $(Q^{\mu}_{\mathcal{T \cup N}})^TX^{\mu}Q^{\mu}_{\mathcal{T \cup N}}$. Therefore, if $\Xi_X^{\mu}(\epsilon)$ is so small that~\eqref{condition_on_perturbation_X} holds, then there exist, see Theorem 8.1.10 in~\cite{GL13}, $V^{\mu}_{\epsilon} \in \mathbb{R}^{(n-n_{\mathcal{B}}) \times n_{\mathcal{B}}}$ and an orthogonal matrix   
\begin{align*}
Y^{\mu}_{\epsilon}:=\!
\begingroup 
\setlength\arraycolsep{.8pt}
\begin{pmatrix} I_{n_{\mathcal{B}}} & -(V^{\mu}_{\epsilon})^T\\V^{\mu}_{\epsilon} & I_{n-n_{\mathcal{B}}} \end{pmatrix} \endgroup 
\begingroup 
\setlength\arraycolsep{.8pt} \begin{pmatrix} \big(I_{n_{\mathcal{B}}}+(V^{\mu}_{\epsilon})^T V^{\mu}_{\epsilon}\big)^{-\frac12} & 0\\0 & \big(I_{n-n_{\mathcal{B}}}+V^{\mu}_{\epsilon} (V^{\mu}_{\epsilon})^T\big)^{-\frac12} \end{pmatrix} \endgroup
\end{align*}
in which
\begin{align}
\big\|V^{\mu}_{\epsilon}\big\| &\le \frac{4 \big \|(Q^{\mu}_{\mathcal{B}})^T \Xi_X^{\mu}(\epsilon) Q^{\mu}_{\mathcal{T} \cup \mathcal{N}} \big \|}{\lambda_{[n_{\mathcal{B}}]}(X^{\mu})-\lambda_{[n_{\mathcal{B}}+1]}(X^{\mu})}, \label{upper_bound_perturbation} 
\end{align} 
such that the first $n_{\mathcal{B}}$ columns of $Q^{\mu} Y^{\mu}_{\epsilon}$ form an orthonormal basis for an invariant subspace of $X^{\mu}(\epsilon)$. In other words, we get
\begin{align}
(Q^{\mu} Y^{\mu}_{\epsilon})^T X^{\mu}(\epsilon) Q^{\mu} Y^{\mu}_{\epsilon} = \begin{pmatrix} D^{\mu}_{\mathcal{B}(\epsilon)} & 0\\0 & D^{\mu}_{\mathcal{T(\epsilon) \cup N(\epsilon)}} \end{pmatrix}, \label{invariant_subspace_perturbed}
\end{align}
where $D^{\mu}_{\mathcal{B}(\epsilon)} \in \mathbb{S}^{n_{\mathcal{B}}}$ and $D^{\mu}_{\mathcal{T(\epsilon) \cup N(\epsilon)}} \in \mathbb{S}^{n-n_{\mathcal{B}}}$ are positive definite matrices. The eigenvalues of $D^{\mu}_{\mathcal{B}(\epsilon)}$ and $D^{\mu}_{\mathcal{T(\epsilon) \cup N(\epsilon)}}$ are equal to those of 
\begin{align*}
&(Q^{\mu}_{\mathcal{B}})^T X^{\mu}Q^{\mu}_{\mathcal{B}} + (Q^{\mu}_{\mathcal{B}})^T\Xi_X^{\mu}(\epsilon) Q^{\mu}_{\mathcal{B}} + (Q^{\mu}_{\mathcal{B}})^T \Xi_X^{\mu}(\epsilon) Q^{\mu}_{\mathcal{T} \cup \mathcal{N}} V^{\mu}_{\epsilon},\\
 &(Q^{\mu}_{\mathcal{T \cup N}})^TX^{\mu}Q^{\mu}_{\mathcal{T \cup N}} + (Q^{\mu}_{\mathcal{T \cup N}})^T\Xi_X^{\mu}(\epsilon) Q^{\mu}_{\mathcal{T \cup N}} - V^{\mu}_{\epsilon}(Q^{\mu}_{\mathcal{B}})^T \Xi_X^{\mu}(\epsilon) Q^{\mu}_{\mathcal{T} \cup \mathcal{N}}, 
 \end{align*}
respectively, see Theorem 4.12 in~\cite{St73}. Condition~\eqref{perturbed_threshold} allows for the identification of $\big(Q_{\mathcal{B}(\epsilon)}^{\mu},Q_{\mathcal{T}(\epsilon)}^{\mu},Q_{\mathcal{N}(\epsilon)}^{\mu}\big)$. On the other hand, conditions~\eqref{condition_on_perturbation_X} and~\eqref{upper_bound_perturbation} guarantee that 
\begin{align*}
\lambda_{\min}\big(D^{\mu}_{\mathcal{B}(\epsilon)}\big) >  \lambda_{\max}\big(D^{\mu}_{\mathcal{T(\epsilon) \cup N(\epsilon)}}\big) > 0,
\end{align*}
i.e., the eigenvalues of $D^{\mu}_{\mathcal{B}(\epsilon)}$ and $D^{\mu}_{\mathcal{T(\epsilon) \cup N(\epsilon)}}$ are properly separated at $\mu$. Consequently, we can conclude that the first $n_{\mathcal{B}}$ columns of $Q^{\mu} Y^{\mu}_{\epsilon}$ form an orthonormal basis for $\mathcal{R}\big(Q^{\mu}_{\mathcal{B}(\epsilon)}\big)$. More precisely, let $M^{\mu}_{\epsilon}=\big(M_{\mathcal{B}(\epsilon)}^{\mu} \ M_{\mathcal{T(\epsilon) \cup N(\epsilon)}}^{\mu}\big):=Q^{\mu} Y^{\mu}_{\epsilon}$. Then we have from~\eqref{invariant_subspace_perturbed} that
\begin{align*}
X^{\mu}(\epsilon) &= M^{\mu}_{\mathcal{B}(\epsilon)} D^{\mu}_{\mathcal{B}(\epsilon)} \big(M^{\mu}_{\mathcal{B}(\epsilon)}\big)^T + M^{\mu}_{\mathcal{T(\epsilon) \cup N(\epsilon)}} D^{\mu}_{\mathcal{T(\epsilon) \cup N(\epsilon)}} \big(M^{\mu}_{\mathcal{T(\epsilon) \cup N(\epsilon)}}\big)^T\\
&=M^{\mu}_{\mathcal{B}(\epsilon)} P_{\epsilon} \Lambda \big(D^{\mu}_{\mathcal{B}(\epsilon)}\big) \big(M^{\mu}_{\mathcal{B}(\epsilon)} P_{\epsilon}\big)^T  + M^{\mu}_{\mathcal{T(\epsilon) \cup N(\epsilon)}} P'_{\epsilon} \Lambda\big(D^{\mu}_{\mathcal{T(\epsilon) \cup N(\epsilon)}}\big) \big(M^{\mu}_{\mathcal{T(\epsilon) \cup N(\epsilon)}} P'_{\epsilon}\big)^T,
\end{align*}
where $P_{\epsilon} \in \mathbb{R}^{n_{\mathcal{B}} \times n_{\mathcal{B}}}$ and $P'_{\epsilon} \in \mathbb{R}^{(n-n_{\mathcal{B}}) \times (n-n_{\mathcal{B}})}$ are orthogonal matrices. All this implies that 
\begin{align*}
\mathcal{R}\big(Q^{\mu}_{\mathcal{B}(\epsilon)}\big)=\mathcal{R}\big(M^{\mu}_{\mathcal{B}(\epsilon)} P_{\epsilon}\big)=\mathcal{R}\big(M^{\mu}_{\mathcal{B}(\epsilon)}\big).
\end{align*}
The distance between $\mathcal{R}\big(Q^{\mu}_{\mathcal{B}}\big)$ and $\mathcal{R}\big(Q^{\mu}_{\mathcal{B}(\epsilon)}\big)$ is the result of Corollary 8.1.11 in~\cite{GL13}. This completes the proof. 
\end{proof}
%%%%%%%%
%New Remark
%%%%%%%%
\begin{remark}\label{convergence_subspaces}
Interestingly, Theorem~\ref{Stewart_eigenspace_theorem} can be modified to quantify the proximity of $\mathcal{R}\big(Q_{\mathcal{B}}^{\mu}\big)$ and $\mathcal{R}\big(Q_{\mathcal{N}}^{\mu}\big)$ to the subspaces $\mathcal{B}$ and $\mathcal{N}$. Let $Q=(Q_{\mathcal{B}},Q_{\mathcal{T}},Q_{\mathcal{N}})$ be an orthonormal basis partitioned according to the optimal partition at $\epsilon=0$, and let $Y^{\mu}$ and $M^{\mu}_{\mathcal{B}}$ be defined as in the proof of Theorem~\ref{Stewart_eigenspace_theorem}, in which $Q^{\mu}$ is replaced by $Q$. Then it is easy to verify that $\distance\!\big(\mathcal{R}\big(Q^{\mu}_{\mathcal{B}}\big), \mathcal{B}\big) \to 0$ for any sequence $\{\mu\} \downarrow 0$.
\end{remark}
%%%%%%%%
%New Remark
%%%%%%%%
\begin{remark}
In the proof of Theorem~\ref{Stewart_eigenspace_theorem}, if $\epsilon$ is fixed and so small that $V^{\mu}_{\epsilon}$ exists for every $0 \le \mu < \tilde{\mu}$, then for any sequence $\{\mu_k\} \downarrow 0$ there exists $M^{\mu_k}_{\mathcal{B}(\epsilon)}$ such that 
\begin{align*}
\mathcal{R}\big(M^{\mu_k}_{\mathcal{B}(\epsilon)}\big)=\mathcal{R}\big(Q^{\mu_k}_{\mathcal{B}(\epsilon)}\big),
\end{align*}
when $k$ is sufficiently large. Therefore, by Remark~\ref{convergence_subspaces} and the triangle inequality, we get
\begin{align*}
\distance\!\big(\mathcal{R}\big(M^{\mu_k}_{\mathcal{B}(\epsilon)}\big), \mathcal{B}(\epsilon)\big) \le \distance\!\big(\mathcal{R}\big(M^{\mu_k}_{\mathcal{B}(\epsilon)}\big), \mathcal{R}\big(Q^{\mu_k}_{\mathcal{B}(\epsilon)}\big)\big) + \distance\!\big(\mathcal{R}\big(Q^{\mu_k}_{\mathcal{B}(\epsilon)}\big), \mathcal{B}(\epsilon)\big) \to 0,
\end{align*}
and thus the columns of an accumulation point of $M^{\mu_k}_{\mathcal{B}(\epsilon)}$ form an orthonormal basis for $\mathcal{B}(\epsilon)$. The case for $\mathcal{N}(\epsilon)$ is analogous.
\end{remark}

\vspace{5px}
\noindent
Notice that~\eqref{estimated_distance_B} and \eqref{estimated_distance_N} reflect the sensitivity of the approximation of the optimal partition in a neighborhood of $0$, when $\epsilon$ belongs to a nonlinearity interval. However, the application of Theorem~\ref{Stewart_eigenspace_theorem} requires an estimate of the effect of the perturbation on the central solutions. Due to the nonsingularity of the Jacobian, an upper bound on $\|\Xi_X^{\mu}(\epsilon)\|$ and $\|\Xi_S^{\mu}(\epsilon)\|$ can be obtained by using the Kantorovich theorem, see e.g., Theorem 5.3.1 in~\cite{DS83}. 
%%%%%%%%
%New Theorem
%%%%%%%% 
\begin{theorem}[Theorem 5.3.1 in~\cite{DS83}]\label{kantorovich}
Given a solution $x_0 \in \mathbb{R}^n$, let $G:\mathbb{R}^n \to \mathbb{R}^n$ be a continuously differentiable mapping on $\|x-x_0\|_2 \le r$. Assume that $\nabla G(x_0)$ is nonsingular and Lipschitz continuous with Lipschitz constant $\tau$ on $\|x-x_0\|_2 \le r$. Furthermore, define  
\begin{align*}
\theta:=\big \|\nabla G^{-1}(x_0) \big \|_2, \qquad \eta:=\big \|\nabla G^{-1}(x_0)G(x_0) \big \|_2.
\end{align*}
If $\tau\theta \eta \le \frac 12$ and $(1-\sqrt{1-2\tau\theta\eta})/(\theta \tau) \le r$, then there exists a solution $x^*$ to $G(x) = 0$ such that 
\begin{align*}
\|x^*-x_0\|_2 \le \frac{1-\sqrt{1-2\tau\theta\eta}}{\theta \tau}.
\end{align*}
\qed
\end{theorem}
\noindent
Now, we can apply Kantorovich theorem to $F$, as defined in~\eqref{alternative_CP}. To that end, we define
\begin{equation}\label{Kant_attributes}
\begin{aligned}
\delta^{\mu}&:=\min\!\Big \{\lambda_{[n_{\mathcal{B}}+n_{\mathcal{T}}]}(X^{\mu}), \ \lambda_{[n_{\mathcal{N}}]}(S^{\mu}) \Big\},\\
\theta^{\mu}&:=\big\|J^{-1} \big(X^{\mu},y^{\mu},S^{\mu}\big)\big \|_2,\\
\eta^{\mu}&:=\big \|J^{-1}\big(X^{\mu},y^{\mu},S^{\mu}\big) F\big(X^{\mu},y^{\mu},S^{\mu},\mu,\epsilon\big) \big\|_2.
\end{aligned}
\end{equation}
 %%%%%%%
%New Lemma
 %%%%%%%
\begin{lemma}\label{bound_on_Delta}
Let $\big(X^{\mu},y^{\mu},S^{\mu}\big)$ be a central solution. If $\epsilon$ is chosen in such a way that
\begin{align}
|\epsilon| < \min\!\Bigg \{\frac{\delta^{\mu}}{2\theta^{\mu}\big\|\bar{C}\big\|}, \  \frac{1}{2(\theta^{\mu})^2\big\|\bar{C}\big\|} \Bigg \},  \label{Kant_bound}
\end{align}
then there exists a central solution $\big(X^{\mu}(\epsilon),y^{\mu}(\epsilon),S^{\mu}(\epsilon)\big)$ such that 
\begin{align}
\big \|\Xi_X^{\mu}(\epsilon) \big \| &\le \frac{1-\sqrt{1-2|\epsilon| (\theta^{\mu})^2 \big\|\bar{C}\big\|}}{\theta^{\mu}}, \label{upper_bound_X}\\[-1\jot]
\big \|\Xi_S^{\mu}(\epsilon) \big \| &\le  \frac{1-\sqrt{1-2|\epsilon| (\theta^{\mu})^2 \big\|\bar{C}\big\|}}{\theta^{\mu}} \nonumber.
\end{align}
\end{lemma}
\begin{proof}
Note that $F$ is continuously differentiable, and $J$ is Lipschitz continuous with global Lipschitz constant 1, see Lemma 2 in~\cite{NO99}. Furthermore, we have
\begin{align*}
\eta^{\mu} \le \big \|J^{-1}\big(X^{\mu},y^{\mu},S^{\mu}\big) \big \|_2 \big \|F\big(X^{\mu},y^{\mu},S^{\mu},\mu,\epsilon\big) \big\|_2 = |\epsilon| \theta^{\mu} \big \|\bar{C} \big \|,
\end{align*}
where the last equality follows from
\begin{align*}
F\big(X^{\mu},y^{\mu},S^{\mu},\mu,\epsilon\big)=\begin{pmatrix} 0 \\ -\epsilon\svectorize(\bar{C}) \\ 0 \end{pmatrix}.
\end{align*}
Thus, by the condition of Kantorovich theorem, if 
\begin{align*}
\eta^{\mu}\theta^{\mu} \le |\epsilon| (\theta^{\mu})^2 \big\|\bar{C}\big\| \le \frac12,
\end{align*}
then there exists an $\big(\hat{X}^{\mu}(\epsilon),\hat{y}^{\mu}(\epsilon),\hat{S}^{\mu}(\epsilon)\big)$ satisfying the equations in~\eqref{alternative_CP}, such that 
\begin{align*}
\big \|\big(\svectorize(\hat{X}^{\mu}(\epsilon)-X^{\mu}); \ \hat{y}^{\mu}(\epsilon)-y^{\mu}; \ \svectorize(\hat{S}^{\mu}(\epsilon)-S^{\mu})\big) \big \|_2 \le \frac{1-\sqrt{1-2|\epsilon| (\theta^{\mu})^2 \big\|\bar{C}\big\|}}{\theta^{\mu}}.
\end{align*}
\noindent
In particular, this implies that for $i=1,\ldots,n_{\mathcal{B}}+n_{\mathcal{T}}$ 
\begin{align}\label{condition_positivity_X}
\big |\lambda_{[i]}\big(\hat{X}^{\mu}(\epsilon)\big)-\lambda_{[i]}\big(X^{\mu}\big) \big | \le \big \|\hat{X}^{\mu}(\epsilon)-X^{\mu} \big \| \le \frac{1-\sqrt{1-2|\epsilon| (\theta^{\mu})^2 \big\|\bar{C}\big\|}}{\theta^{\mu}},
\end{align}
and that for $j=1,\ldots,n_{\mathcal{N}}$
\begin{align}\label{condition_positivity_S}
\big |\lambda_{[j]}\big(\hat{S}^{\mu}(\epsilon)\big)-\lambda_{[j]}\big(S^{\mu}\big) \big | \le \big \|\hat{S}^{\mu}(\epsilon)-S^{\mu} \big \| \le \frac{1-\sqrt{1-2|\epsilon| (\theta^{\mu})^2 \big\|\bar{C}\big\|}}{\theta^{\mu}}.
\end{align}
On the other hand, $\hat{X}^{\mu}(\epsilon)$ and $\hat{S}^{\mu}(\epsilon)$ stay positive definite if
\begin{align*}
\big |\lambda_{[i]}\big(\hat{X}^{\mu}(\epsilon)\big)-\lambda_{[i]}\big(X^{\mu}\big) \big | &< \delta^{\mu}, & i&=1,\ldots,n_{\mathcal{B}}+n_{\mathcal{T}},\\
\big |\lambda_{[j]}\big(\hat{S}^{\mu}(\epsilon)\big)-\lambda_{[j]}\big(S^{\mu}\big) \big | &< \delta^{\mu}, & j&=1,\ldots,n_{\mathcal{N}},
\end{align*}
which together with~\eqref{condition_positivity_X} and~\eqref{condition_positivity_S} induce the following bound:  
\begin{align}
\delta^{\mu} > \frac{1-\big(1-2|\epsilon|(\theta^{\mu})^2 \big\|\bar{C}\big\|\big)}{\theta^{\mu}} \ge  \frac{1-\sqrt{1-2|\epsilon| (\theta^{\mu})^2 \big\|\bar{C}\big\|}}{\theta^{\mu}}, \label{positive_condition}
\end{align}
where the second inequality in~\eqref{positive_condition} follows from $2|\epsilon| (\theta^{\mu})^2 \big\|\bar{C}\big\| \le 1$. Note that if~\eqref{positive_condition} holds, then $\lambda_{[i]}\big(\hat{X}^{\mu}(\epsilon)\big) > 0$ for $i=n-n_{\mathcal{N}} + 1,\ldots,n$ and $\lambda_{[j]}\big(\hat{S}^{\mu}(\epsilon)\big) > 0$ for $j=n_{\mathcal{N}} + 1,\ldots,n$ are immediate from~\eqref{central_path_equations}. Consequently, if~\eqref{Kant_bound} holds, then solution $\big(\hat{X}^{\mu}(\epsilon),\hat{y}^{\mu}(\epsilon),\hat{S}^{\mu}(\epsilon)\big)$ satisfies~\eqref{alternative_CP}, and it is indeed a central solution for the perturbed SDO problem. The proof is complete.
\end{proof}

\vspace{5px}
\noindent
Using the results of Lemma~\ref{bound_on_Delta}, we can now derive upper bounds on the distance of $\mathcal{R}\big(Q^{\mu}_{\mathcal{B}(\epsilon)}\big)$ and $\mathcal{R}\big(Q^{\mu}_{\mathcal{N}(\epsilon)}\big)$ from $\mathcal{R}\big(Q^{\mu}_{\mathcal{B}}\big)$ and $\mathcal{R}\big(Q^{\mu}_{\mathcal{N}}\big)$, respectively.
%%%%%%%%
%New Theorem
%%%%%%%% 
\begin{theorem}\label{existence_Q_central_solutions}
Let a central solution $\big(X^{\mu},y^{\mu},S^{\mu}\big)$ be given, and let $\epsilon$ belong to a nonlinearity interval in a neighborhood of $0$. If $\mu < \min\{\tilde{\mu},\tilde{\mu}(\epsilon)\}$ and
\begin{align*}
|\epsilon| < \frac{1}{\theta^{\mu}\big\|\bar{C}\big\|}\min\!\Bigg \{\frac{\delta^{\mu}}{2}, \  \frac{1}{2\theta^{\mu}}, \ \frac{\lambda_{[n_{\mathcal{B}}]}(X^{\mu})-\lambda_{[n_{\mathcal{B}}+1]}(X^{\mu})}{10}, \ \frac{\lambda_{[n_{\mathcal{N}}]}(S^{\mu})-\lambda_{[n_{\mathcal{N}}+1]}(S^{\mu})}{10} \Bigg \}
\end{align*}
hold, then there exists $\big(X^{\mu}(\epsilon),y^{\mu}(\epsilon),S^{\mu}(\epsilon)\big)$ such that  
\begin{align}
\distance\!\big(\mathcal{R}\big(Q^{\mu}_{\mathcal{B}}\big),\mathcal{R}\big(Q^{\mu}_{\mathcal{B}(\epsilon)}\big)\big) \le \frac{4\Big(1-\sqrt{1-2|\epsilon|(\theta^{\mu})^2 \big\|\bar{C}\big\|}\Big)}{\theta^{\mu} \big(\lambda_{[n_{\mathcal{B}}]}(X^{\mu})-\lambda_{[n_{\mathcal{B}}+1]}(X^{\mu})\big)}, \label{distance_between_subspaces_B}\\
\distance\!\big(\mathcal{R}\big(Q^{\mu}_{\mathcal{N}}\big),\mathcal{R}\big(Q^{\mu}_{\mathcal{N}(\epsilon)}\big)\big) \le \frac{4\Big(1-\sqrt{1-2|\epsilon|(\theta^{\mu})^2 \big\|\bar{C}\big\|}\Big)}{\theta^{\mu}\big(\lambda_{[n_{\mathcal{N}}]}(S^{\mu})-\lambda_{[n_{\mathcal{N}}+1]}(S^{\mu})\big)}. \label{distance_between_subspaces_N}
\end{align}
\end{theorem}
\begin{proof}
Condition~\eqref{condition_on_perturbation_X}, after including~\eqref{upper_bound_X}, holds if 
\begin{align*}
\frac{1-\sqrt{1-2|\epsilon|(\theta^{\mu})^2 \big\|\bar{C}\big\|}}{\theta^{\mu}} \le \frac{1-(1-2|\epsilon|(\theta^{\mu})^2 \big\|\bar{C}\big\|)}{\theta^{\mu}} \le \frac{\lambda_{[n_{\mathcal{B}}]}(X^{\mu})-\lambda_{[n_{\mathcal{B}}+1]}(X^{\mu})}{5},  
\end{align*}
which gives the upper bound
\begin{align*}
 |\epsilon| \le \frac{\lambda_{[n_{\mathcal{B}}]}(X^{\mu})-\lambda_{[n_{\mathcal{B}}+1]}(X^{\mu})}{10\theta^{\mu}\big\|\bar{C}\big\|}.
\end{align*}
The upper bounds on the distance between the subspaces are immediate from~\eqref{estimated_distance_B} and~\eqref{estimated_distance_N}.
\end{proof}

%%%%%%%%
%New Section
%%%%%%%%
\section{Numerical experiments}\label{experiments}
In this section, we investigate the sensitivity of the central path and the optimal partition on Example~\ref{motivation_nonlinearity}. Recall that $(-\infty,-\frac12)$ and $(\frac32,\infty)$ are invariancy intervals, $(-\frac12,\frac32)$ is a nonlinearity interval, and $-\frac12$ as well as $\frac32$ are the transition points. The strict complementarity condition fails only at $-\frac12$ and $\frac32$, the dual nondegeneracy condition fails only at $\frac12$, and the eigenvalues of unique primal optimal solutions at $0$ and $1$ are of multiplicity 2.

 \begin{figure}[H]
\begin{center}
\includegraphics[height=2.0in]{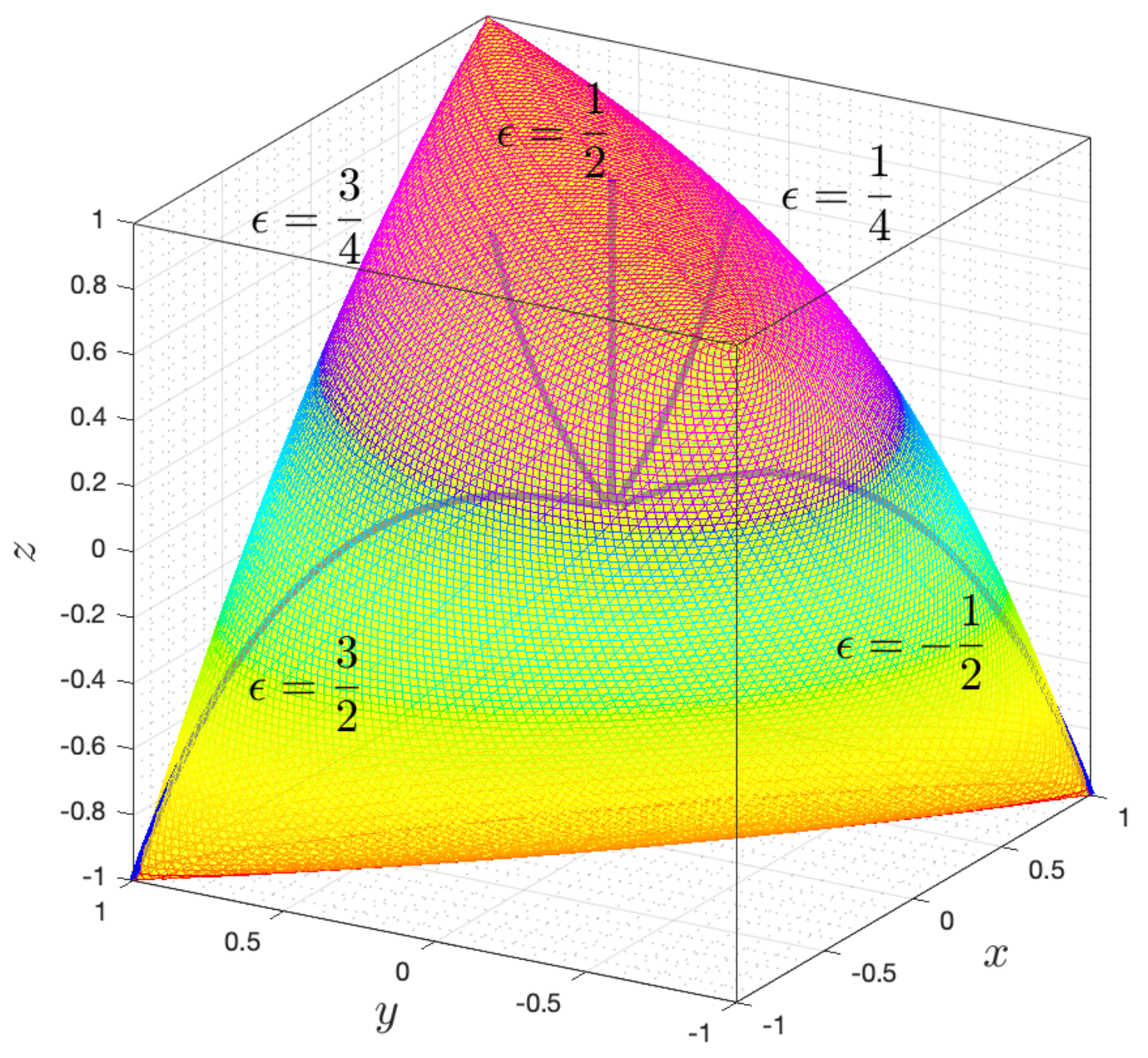}
\caption{The central path for different values of $\epsilon$.}
\label{fig:3elliptope_central_path}
\end{center}
\end{figure}

\vspace{5px}
\noindent
Tables~\ref{threshold} through~\ref{sensitivity_simple} and Figure~\ref{fig:3elliptope_central_path}  represent a summary of numerical experiments, where $\big(X^a(\epsilon),y^a(\epsilon),S^a(\epsilon)\big)$ is the analytic center of the optimal set and $\breve{\mu}(\epsilon)$ denotes a lower bound on the largest $\mu$ which allows for the identification of $\big(Q_{\mathcal{B}(\epsilon)}^{\mu},Q_{\mathcal{T}(\epsilon)}^{\mu},Q_{\mathcal{N}(\epsilon)}^{\mu}\big)$. To numerically obtain $\breve{\mu}(\epsilon)$, initially set to $1$, $\mu$ is sequentially decreased at a geometric rate $0.9$ until the eigenvalues with positive limit points of $X^{\mu}(\epsilon)$ and $S^{\mu}(\epsilon)$ can be correctly identified, up to a certain precision. In our experiments, the limit point of an eigenvalue of $X^{\mu}(\epsilon)$ or $S^{\mu}(\epsilon)$ is taken as $0$ if the eigenvalue drops below $10^{-5}$.

\begin{table}[H]
\small
\caption{The numerical behavior of the central solutions at different values of $\epsilon$.}
\label{threshold}
\resizebox{\columnwidth}{!}{
\begin{tabular}{cccccc}
\hline
$\epsilon$ & $\breve{\mu}(\epsilon)$       & $\distance\!\Big(\mathcal{B}(\epsilon),\mathcal{R}\big(Q_{\mathcal{B}(\epsilon)}^{\breve{\mu}(\epsilon)}\big)\Big)$ & $\distance\!\Big(\mathcal{N}(\epsilon),\mathcal{R}\big(Q_{\mathcal{N}(\epsilon)}^{\breve{\mu}(\epsilon)}\big)\Big)$ & $\big\|X^{\breve{\mu}(\epsilon)}(\epsilon)-X^{a}(\epsilon)\big\|$  & $\big\|S^{\breve{\mu}(\epsilon)}(\epsilon)-S^{a}(\epsilon)\big\|$  \\
\hline
-1   & 9.953E-06 & 1.043E-06 & 1.043E-06 & 1.556E-05 & 1.724E-05 \\
-0.75   & 4.975E-06 & 1.094E-06 & 1.094E-06 & 1.528E-05 & 1.450E-05 \\
-0.50   & 4.951E-11 & 1.175E-06 & 1.175E-06 & 1.495E-05 & 1.221E-05 \\
-0.25   & 8.734E-06 & 1.465E-06 & 1.465E-06 & 1.226E-05 & 1.433E-05 \\
0    & 1.488E-05 & 3.485E-16 & 4.328E-16 & 6.074E-06 & 1.718E-05 \\
0.25    & 1.123E-05 & 6.273E-07 & 6.273E-07 & 7.042E-06 & 1.347E-05 \\
0.50    & 9.953E-06 & 0.000E+00 & 0.000E+00 & 7.037E-06 & 1.219E-05 \\
0.75    & 1.123E-05 & 6.273E-07 & 6.273E-07 & 7.042E-06 & 1.347E-05 \\
1    & 1.488E-05 & 3.485E-16 & 4.328E-16 & 6.074E-06 & 1.718E-05 \\
1.25    & 8.734E-06 & 1.465E-06 & 1.465E-06 & 1.226E-05 & 1.433E-05 \\
1.50    & 4.951E-11 & 1.175E-06 & 1.175E-06 & 1.495E-05 & 1.221E-05 \\
1.75    & 4.975E-06 & 1.094E-06 & 1.094E-06 & 1.528E-05 & 1.450E-05 \\
2    & 9.953E-06 & 1.043E-06 & 1.043E-06 & 1.556E-05 & 1.724E-05\\
\hline
\end{tabular}}
\end{table}

\vspace{5px}
\noindent
In Table~\ref{threshold}, we show how small $\mu$ should approximately be in order to identify $\big(Q_{\mathcal{B}(\epsilon)}^{\mu},Q_{\mathcal{T}(\epsilon)}^{\mu},Q_{\mathcal{N}(\epsilon)}^{\mu}\big)$. We also highlight the proximity of the central solutions and the approximation of the optimal partition once $\big(Q_{\mathcal{B}(\epsilon)}^{\mu},Q_{\mathcal{T}(\epsilon)}^{\mu},Q_{\mathcal{N}(\epsilon)}^{\mu}\big)$ is identified. One can observe that $\breve{\mu}(\epsilon)$ gets comparatively smaller values at $-\frac12$ and $\frac32$, where the strict complementarity condition fails. Further, $\mathcal{R}\big(Q_{\mathcal{B}(\epsilon)}^{\breve{\mu}(\epsilon)}\big)$ and $\mathcal{R}\big(Q_{\mathcal{N}(\epsilon)}^{\breve{\mu}(\epsilon)}\big)$ are in close proximity to the true optimal partition at $0$, $1$, and $\frac12$, in spite of multiplicity of the eigenvalues or failure of the dual nondegeneracy condition.

\vspace{5px}
\noindent
At fixed $\epsilon=-\frac12$ and $\epsilon=0$, Tables~\ref{Decrease_SC_fails} and~\ref{Decrease_SC_holds} demonstrate the convergence of $\mathcal{R}\big(Q_{\mathcal{B}(\epsilon)}^{\mu}\big)$, $\mathcal{R}\big(Q_{\mathcal{N}(\epsilon)}^{\mu}\big)$, and $\big(X^{\mu}(\epsilon),y^{\mu}(\epsilon),S^{\mu}(\epsilon)\big)$ to $\mathcal{B}(\epsilon)$, $\mathcal{N}(\epsilon)$, and the analytic center of the optimal set, respectively. At $\epsilon = -\frac12$, $\distance\!\big(\mathcal{B}(\epsilon),\mathcal{R}\big(Q_{\mathcal{B}(\epsilon)}^{\mu}\big)\big)$ and $\|X^{\mu}(\epsilon)-X^{a}(\epsilon)\|$ converge at almost the same rate, and they are of approximate order $\mathcal{O}(\sqrt{\mu})$. Analogous results can be observed for $\distance\!\big(\mathcal{N}(\epsilon),\mathcal{R}\big(Q_{\mathcal{N}(\epsilon)}^{\mu}\big)\big)$ and $\|S^{\mu}(\epsilon)-S^{a}(\epsilon)\|$.  At $\epsilon = 0$, on the other hand, $\distance\!\big(\mathcal{B}(\epsilon),\mathcal{R}\big(Q_{\mathcal{B}(\epsilon)}^{\mu}\big)\big)$ and $\distance\!\big(\mathcal{N}(\epsilon),\mathcal{R}\big(Q_{\mathcal{N}(\epsilon)}^{\mu}\big)\big)$ converge faster at the beginning, but they become very slow ultimately. In this case, as expected from~\cite{LS98}, $X^{\mu}(\epsilon)$ and $S^{\mu}(\epsilon)$ stay in $\mathcal{O}(\mu)$ proximity of the analytic center of the optimal set.

\begin{table}[H]
\small
\centering
\caption{Convergence to the analytic center and the true optimal partition at $\epsilon = -\frac12$.}
\label{Decrease_SC_fails}
\resizebox{\columnwidth}{!}{
\begin{tabular}{ccccc}
\hline
$\mu$       & $\distance\!\big(\mathcal{B}(\epsilon),\mathcal{R}\big(Q_{\mathcal{B}(\epsilon)}^{\mu}\big)\big)$ & $\distance\!\big(\mathcal{N}(\epsilon),\mathcal{R}\big(Q_{\mathcal{N}(\epsilon)}^{\mu}\big)\big)$ & $\|X^{\mu}(\epsilon)-X^{a}(\epsilon)\|$  & $\|S^{\mu}(\epsilon)-S^{a}(\epsilon)\|$  \\
\hline
1.E-11 & 5.280E-07 & 5.280E-07 & 6.720E-06 & 5.487E-06 \\
1.E-12 & 1.670E-07 & 1.670E-07 & 2.125E-06 & 1.735E-06 \\
1.E-13 & 5.282E-08 & 5.282E-08 & 6.723E-07 & 5.490E-07 \\
1.E-14 & 1.679E-08 & 1.679E-08 & 2.137E-07 & 1.745E-07 \\
1.E-15 & 5.038E-09 & 5.038E-09 & 6.413E-08 & 5.236E-08 \\
1.E-16 & 6.249E-09 & 6.249E-09 & 7.954E-08 & 6.494E-08\\
\hline
\end{tabular}}
\end{table}

\begin{table}[H]
\small
\centering
\caption{Convergence to the analytic center and the true optimal partition at $\epsilon = 0$.}
\label{Decrease_SC_holds}
\resizebox{\columnwidth}{!}{
\begin{tabular}{ccccc}
\hline
$\mu$  & $\distance\!\big(\mathcal{B}(\epsilon),\mathcal{R}\big(Q_{\mathcal{B}(\epsilon)}^{\mu}\big)\big)$   & $\distance\!\big(\mathcal{N}(\epsilon),\mathcal{R}\big(Q_{\mathcal{N}(\epsilon)}^{\mu}\big)\big)$  & $\|X^{\mu}(\epsilon)-X^{a}(\epsilon)\|$    & $\|S^{\mu}(\epsilon)-S^{a}(\epsilon)\|$    \\
\hline
1.E-05 & 5.327E-16 & 4.328E-16 & 4.082E-06 & 1.155E-05 \\
1.E-06 & 4.937E-16 & 4.328E-16 & 4.082E-07 & 1.155E-06 \\
1.E-07 & 2.878E-16 & 4.328E-16 & 4.082E-08 & 1.155E-07 \\
1.E-08 & 4.600E-16 & 4.328E-16 & 4.082E-09 & 1.155E-08 \\
1.E-09 & 5.849E-16 & 4.328E-16 & 4.082E-10 & 1.155E-09 \\
1.E-10 & 5.087E-16 & 4.328E-16 & 4.082E-11 & 1.155E-10 \\
1.E-11 & 5.660E-16 & 4.328E-16 & 4.082E-12 & 1.155E-11 \\
1.E-12 & 9.407E-16 & 4.328E-16 & 4.083E-13 & 1.155E-12 \\
1.E-13 & 5.373E-16 & 4.328E-16 & 4.084E-14 & 1.155E-13 \\
1.E-14 & 7.640E-16 & 4.328E-16 & 4.081E-15 & 1.154E-14 \\
1.E-15 & 4.686E-16 & 6.958E-16 & 4.578E-16 & 1.154E-15 \\
1.E-16 & 2.373E-16 & 4.328E-16 & 1.110E-16 & 3.140E-16\\
\hline
\end{tabular}}
\end{table}

\vspace{5px}
\noindent
In Tables~\ref{sensitivity_multiplicity} and~\ref{sensitivity_simple}, we investigate the sensitivity of the approximation of the optimal partition around $\epsilon=0$ and $\epsilon=\frac12$ at a small enough fixed $\mu$ which allows for the identification of $\big(Q_{\mathcal{B}(\epsilon')}^{\mu},Q_{\mathcal{T}(\epsilon')}^{\mu},Q_{\mathcal{N}(\epsilon')}^{\mu}\big)$. We can observe from the numerical results that the actual values of the distances between the subspaces closely imitate the upper bounds~\eqref{estimated_distance_B} and~\eqref{estimated_distance_N}. The graphs of $\distance\!\big(\mathcal{R}\big(Q_{\mathcal{B}(\epsilon)}^{\mu}\big),\mathcal{R}\big(Q_{\mathcal{B}(\epsilon')}^{\mu}\big)\big)$ and $\distance\!\big(\mathcal{R}\big(Q_{\mathcal{N}(\epsilon)}^{\mu}\big),\mathcal{R}\big(Q_{\mathcal{N}(\epsilon')}^{\mu}\big)\big)$ versus $\epsilon'$ have an almost symmetric shape, and they reflect a nonsmooth behavior at $0$ and $\frac12$. Furthermore, $\distance\!\big(\mathcal{R}\big(Q_{\mathcal{B}(\frac12)}^{\mu}\big),\mathcal{R}\big(Q_{\mathcal{B}(\epsilon')}^{\mu}\big)\big)$ and $\distance\!\big(\mathcal{R}\big(Q_{\mathcal{N}(\frac12)}^{\mu}\big),\mathcal{R}\big(Q_{\mathcal{N}(\epsilon')}^{\mu}\big)\big)$ vary almost linearly w.r.t. $\epsilon'$.

\begin{table}[H]
\caption{The sensitivity of the approximation of the optimal partition at $\epsilon = 0$.}
\label{sensitivity_multiplicity}
\resizebox{\columnwidth}{!}{
\centering
\begin{tabular}{cccccc}
\hline
$\epsilon'$ & $\mu$       & $\distance\!\big(\mathcal{R}\big(Q_{\mathcal{B}(\epsilon)}^{\mu}\big),\mathcal{R}\big(Q_{\mathcal{B}(\epsilon')}^{\mu}\big)\big)$ & $\distance\!\big(\mathcal{R}\big(Q_{\mathcal{N}(\epsilon)}^{\mu}\big),\mathcal{R}\big(Q_{\mathcal{N}(\epsilon')}^{\mu}\big)\big)$ & The upper bound~\eqref{estimated_distance_B}     & The upper bound~\eqref{estimated_distance_N}     \\
\hline
-0.005 & 1.41E-05 & 4.698E-03 & 4.698E-03 & 1.892E-02 & 1.892E-02 \\
-0.004 & 1.41E-05 & 3.761E-03 & 3.761E-03 & 1.513E-02 & 1.513E-02 \\
-0.003 & 1.41E-05 & 2.823E-03 & 2.823E-03 & 1.134E-02 & 1.134E-02 \\
-0.002 & 1.41E-05 & 1.883E-03 & 1.883E-03 & 7.552E-03 & 7.552E-03 \\
-0.001 & 1.41E-05 & 9.422E-04 & 9.422E-04 & 3.774E-03 & 3.774E-03 \\
0      & 1.41E-05 & 0 & 0 & 0 & 0 \\
0.001  & 1.41E-05 & 9.434E-04 & 9.434E-04 & 3.769E-03 & 3.769E-03 \\
0.002  & 1.41E-05 & 1.888E-03 & 1.888E-03 & 7.532E-03 & 7.532E-03 \\
0.003  & 1.41E-05 & 2.834E-03 & 2.834E-03 & 1.129E-02 & 1.129E-02 \\
0.004  & 1.41E-05 & 3.781E-03 & 3.781E-03 & 1.504E-02 & 1.504E-02 \\
0.005  & 1.41E-05 & 4.730E-03 & 4.730E-03 & 1.879E-02 & 1.879E-02 \\
\hline
\end{tabular}}
\end{table}

\begin{table}[H]
\caption{The sensitivity of the approximation of the optimal partition at $\epsilon = \frac12$.}
\label{sensitivity_simple}
\resizebox{\columnwidth}{!}{
\centering
\begin{tabular}{cccccc}
\hline
$\epsilon'$ & $\mu$       & $\distance\!\big(\mathcal{R}\big(Q_{\mathcal{B}(\epsilon)}^{\mu}\big),\mathcal{R}\big(Q_{\mathcal{B}(\epsilon')}^{\mu})\big)$ & $\distance\!\big(\mathcal{R}\big(Q_{\mathcal{N}(\epsilon)}^{\mu}\big),\mathcal{R}\big(Q_{\mathcal{N}(\epsilon')}^{\mu}\big)\big)$ & The upper bound~\eqref{estimated_distance_B}      & The upper bound~\eqref{estimated_distance_N}     \\
\hline
0.495   & 9.26E-06 & 7.071E-03       & 7.071E-03       & 2.828E-02    & 2.828E-02    \\
0.496   & 9.26E-06 & 5.657E-03       & 5.657E-03       & 2.263E-02    & 2.263E-02    \\
0.497   & 9.26E-06 & 4.243E-03       & 4.243E-03       & 1.697E-02    & 1.697E-02    \\
0.498   & 9.26E-06 & 2.828E-03       & 2.828E-03       & 1.131E-02    & 1.131E-02    \\
0.499   & 9.26E-06 & 1.414E-03       & 1.414E-03       & 5.657E-03    & 5.657E-03    \\
0.5   & 9.26E-06 & 0       & 0       & 0    & 0    \\
0.501   & 9.26E-06 & 1.414E-03       & 1.414E-03       & 5.657E-03    & 5.657E-03    \\
0.502   & 9.26E-06 & 2.828E-03       & 2.828E-03       & 1.131E-02    & 1.131E-02    \\
0.503   & 9.26E-06 & 4.243E-03       & 4.243E-03       & 1.697E-02    & 1.697E-02    \\
0.504   & 9.26E-06 & 5.657E-03       & 5.657E-03       & 2.263E-02    & 2.263E-02    \\
0.505   & 9.26E-06 & 7.071E-03       & 7.071E-03       & 2.828E-02    & 2.828E-02    \\
\hline
\end{tabular}}
\end{table}

%%%%%%%%
%New Section
%%%%%%%%
\section{Concluding remarks and future studies}\label{conclusion}
In this paper, we revisited the parametric analysis and the identification of the optimal partition for SDO problems, when the objective function is perturbed along a fixed direction. We characterized a nonlinearity interval of the optimal partition, where the rank of maximally complementary solutions remain constant, and we provided sufficient conditions for the existence of a nonlinearity interval and a transition point. Additionally, we quantified the sensitivity of the approximation of the optimal partition w.r.t. $\epsilon$ in a nonlinearity interval. Using numerical experiments, we showed how tight the bounds could be for the sensitivity of the approximation of the optimal partition.

\vspace{5px}
\noindent
The continuity and smoothness of optimal solutions on a nonlinearity interval are subjects of future studies. In particular, the limit point of the central path may not be continuous w.r.t. the perturbation of the objective function, see also~\cite{St01}. For instance, a discontinuity is caused by appending a redundant constraint $X_{12} + X_{13} \le 2$ to Example~\ref{motivation_nonlinearity}. While the analytic center of the optimal set at $\epsilon=\frac12$ is given by
\begin{align*}
X^{a}(\frac12)=\begin{pmatrix} \ \ 1& -\frac13 & -\frac13 & 0\\ -\frac13 & \ \ 1 & \ \ 1 & 0\\-\frac13 & \ \ 1 & \ \ 1 & 0 \\ \ \ 0 & \ \ 0 & \ \ 0 & \frac83 \end{pmatrix},
\end{align*}
for any sequence $\{\epsilon_k\} \to \frac12$ we have
\begin{align*}
\lim_{k \to \infty} X^{a}(\epsilon_k)=\begin{pmatrix} 1& 0 & 0 & 0\\ 0 & 1 & 1 & 0\\0 & 1 & 1 & 0 \\ 0 & 0 & 0 & 2 \end{pmatrix}.
\end{align*}

\vspace{5px}
\noindent
Currently, we are investigating theoretical and numerical methods for the computation of a nonlinearity interval.

%%%%%%%%
%New Section
%%%%%%%%
\section*{Acknowledgements}
We are grateful to Professor Rainer Sinn for insightful discussions that resulted in Example~\ref{motivation_nonlinearity}. This work is supported by the Air Force Office of Scientific Research (AFOSR) Grant \# FA9550-15-1-0222.

\bibliographystyle{siam}   % name your BibTeX data base	
\bibliography{optimization}

\begin{thebibliography}{10}

\bibitem{AM92}
{\sc I.~Adler and R.~D.~C. Monteiro}, {\em A geometric view of parametric
  linear programming}, Algorithmica, 8 (1992), pp.~161--176.

\bibitem{AHO97}
{\sc F.~Alizadeh, J.-P.~A. Haeberly, and M.~L. Overton}, {\em Complementarity
  and nondegeneracy in semidefinite programming}, Mathematical Programming, 77
  (1997), pp.~111--128.

\bibitem{AHO98}
{\sc F.~Alizadeh, J.-P.~A. Haeberly, and M.~L. Overton}, {\em Primal-dual
  interior-point methods for semidefinite programming: Convergence rates,
  stability and numerical results}, SIAM Journal on Optimization, 8 (1998),
  pp.~746--768.

\bibitem{BJRT96}
{\sc A.~Berkelaar, B.~Jansen, K.~Roos, and T.~Terlaky}, {\em Sensitivity
  analysis in (degenerate) quadratic programming}, Tech. Rep. 96-26, Delft
  University of Technology, Netherlands, 1996.

\bibitem{BRT97}
{\sc A.~B. Berkelaar, K.~Roos, and T.~Terlaky}, {\em The optimal set and
  optimal partition approach to linear and quadratic programming}, in Advances
  in Sensitivity Analysis and Parametric Programming, T.~Gal and H.~J.
  Greenberg, eds., vol.~6 of International Series in Operations Research \&
  Management Science, Springer, 1997, pp.~159--202.

\bibitem{BS00}
{\sc J.~F. Bonnans and A.~Shapiro}, {\em Perturbation Analysis of Optimization
  Problems}, Springer, 2000.

\bibitem{CSW13}
{\sc Y.-L. Cheung, S.~Schurr, and H.~Wolkowicz}, {\em Preprocessing and
  regularization for degenerate semidefinite programs}, in {Computational and
  Analytical Mathematics, In Honor of {J}onathan {B}orwein's 60th Birthday},
  D.~Bailey, H.~Bauschke, P.~Borwein, F.~Garvan, M.~Th\'era, J.~Vanderwerff,
  and H.~Wolkowicz, eds., vol.~50 of Springer Proceedings in Mathematics \&
  Statistics, Springer, New York, NY, USA, 2013, pp.~613--634.

\bibitem{CW14}
{\sc Y.-L. Cheung and H.~Wolkowicz}, {\em Sensitivity analysis of semidefinite
  programs without strong duality}, tech. rep., 2014.
\newblock \url{http://www.optimization-online.org/DB_HTML/2014/06/4416.html}.

\bibitem{CAPT17}
{\sc D.~Cifuentes, S.~Agarwal, P.~Parrilo, and R.~Thomas}, {\em On the local
  stability of semidefinite relaxations}, 2017.
\newblock arXiv:1710.04287 \url{https://arxiv.org/abs/1710.04287}.

\bibitem{CHS18}
{\sc D.~Cifuentes, C.~Harris, and B.~Sturmfels}, {\em The geometry of
  {SDP}-exactness in quadratic optimization}, 2018.
\newblock arXiv:1804.01796 \url{https://arxiv.org/abs/1804.01796}.

\bibitem{Kl02}
{\sc E.~de~Klerk}, {\em Aspects of Semidefinite Programming: Interior Point
  Algorithms and Selected Applications}, vol.~65 of Series Applied
  Optimization, Springer, 2006.

\bibitem{DS83}
{\sc J.~E. Dennis and R.~B. Schnabel}, {\em Numerical Methods for Unconstrained
  Optimization and Nonlinear Equations}, Prentice-Hall, 1983.

\bibitem{Fi83}
{\sc A.~V. Fiacco}, {\em Introduction to Sensitivity and Stability Analysis in
  Nonlinear Programming}, vol.~165, Academic Press Inc., 1983.

\bibitem{GS99}
{\sc D.~Goldfarb and K.~Scheinberg}, {\em On parametric semidefinite
  programming}, Applied Numerical Mathematics, 29 (1999), pp.~361 -- 377.

\bibitem{GL13}
{\sc G.~H. Golub and C.~F. Van~Loan}, {\em Matrix Computations}, The Johns
  Hopkins University Press, 2013.

\bibitem{G94}
{\sc H.~J. Greenberg}, {\em The use of the optimal partition in a linear
  programming solution for postoptimal analysis}, Operations Research Letters,
  15 (1994), pp.~179 -- 185.

\bibitem{Ha}
{\sc J.-P. Haeberly}, {\em Remarks on nondegeneracy in mixed
  semidefinite-quadratic programming}, 1998.
\newblock Unpublished memorandum, available from
  \url{http://citeseerx.ist.psu.edu/viewdoc/download?doi=10.1.1.43.7501&rep=rep1&type=pdf}.

\bibitem{Hoff52}
{\sc A.~Hoffman}, {\em On approximate solutions of systems of linear
  inequalities}, Journal of Research of the National Bureau of Standards, 49
  (1952), pp.~263--265.

\bibitem{Ho73b}
{\sc W.~W. Hogan}, {\em Point-to-set maps in mathematical programming}, SIAM
  Review, 15 (1973), pp.~591--603.

\bibitem{JRT93}
{\sc B.~Jansen, K.~Roos, and T.~Terlaky}, {\em An interior point method
  approach to postoptimal and parametric analysis in linear programming}, Tech.
  Rep. 92-21, Delft University of Technology, Netherlands, 1993.

\bibitem{LS98}
{\sc Z.-Q. Luo, J.~F. Sturm, and S.~Zhang}, {\em Superlinear convergence of a
  symmetric primal-dual path following algorithm for semidefinite programming},
  SIAM Journal on Optimization, 8 (1998), pp.~59--81.

\bibitem{MT19}
{\sc A.~Mohammad-Nezhad and T.~Terlaky}, {\em On the identification of the
  optimal partition for semidefinite optimisation}, INFOR: Information Systems
  and Operational Research,  (2019).
\newblock doi:10.1080/03155986.2019.1572853.

\bibitem{MT19a}
{\sc A.~Mohammad-Nezhad and T.~Terlaky}, {\em A rounding procedure for
  semidefinite optimization}, Operations Research Letters, 47 (2019), pp.~59 --
  65.

\bibitem{NO99}
{\sc M.~V. Nayakkankuppam and M.~L. Overton}, {\em Conditioning of semidefinite
  programs}, Mathematical Programming, 85 (1999), pp.~525--540.

\bibitem{N10}
{\sc J.~Nie, K.~Ranestad, and B.~Sturmfels}, {\em The algebraic degree of
  semidefinite programming}, Mathematical Programming, 122 (2010),
  pp.~379--405.

\bibitem{FN01}
{\sc M.~Nunez and R.~Freund}, {\em Condition-measure bounds on the behavior of
  the central trajectory of a semidefinite program}, SIAM Journal on
  Optimization, 11 (2001), pp.~818--836.

\bibitem{RD09}
{\sc R.~Rockafellar and A.~Dontchev}, {\em Implicit Functions and Solution
  Mappings}, Springer, 2009.

\bibitem{SW16}
{\sc Y.~Sekiguchi and H.~Waki}, {\em Perturbation analysis of singular
  semidefinite programs and its applications to control problems}, 2016.
\newblock arXiv:1607.05568 \url{https://arxiv.org/abs/1607.05568}.

\bibitem{S97}
{\sc A.~Shapiro}, {\em First and second order analysis of nonlinear
  semidefinite programs}, Mathematical Programming, 77 (1997), pp.~301--320.

\bibitem{St73}
{\sc G.~W. Stewart}, {\em Error and perturbation bounds for subspaces
  associated with certain eigenvalue problems}, SIAM Review, 15 (1973),
  pp.~727--764.

\bibitem{St01}
{\sc J.~Sturm and S.~Zhang}, {\em On sensitivity of central solutions in
  semidefinite programming}, Mathematical Programming, 90 (2001), pp.~205--227.

\bibitem{St99}
{\sc J.~F. Sturm}, {\em Using {SeDuMi} 1.02, {A} {MATLAB} toolbox for
  optimization over symmetric cones}, Optimization Methods and Software, 11
  (1999), pp.~625--653.
\newblock Available at \url{http://sedumi.ie.lehigh.edu/}.

\bibitem{TTT99}
{\sc K.~C. Toh, M.~J. Todd, and R.~H. T\"ut\"unc\"u}, {\em {SDPT3 -- A MATLAB}
  software package for semidefinite programming, version 1.3}, Optimization
  Methods and Software, 11 (1999), pp.~545--581.
\newblock Available at \url{http://www.math.nus.edu.sg/~mattohkc/sdpt3.html}.

\bibitem{TTT03}
{\sc R.~H. T{\"u}t{\"u}nc{\"u}, K.~C. Toh, and M.~J. Todd}, {\em Solving
  semidefinite-quadratic-linear programs using {SDPT3}}, Mathematical
  Programming, 95 (2003), pp.~189--217.

\bibitem{Y2004}
{\sc E.~Yildirim}, {\em Unifying optimal partition approach to sensitivity
  analysis in conic optimization}, Journal of Optimization Theory and
  Applications, 122 (2004), pp.~405--423.

\bibitem{Y2001}
{\sc E.~A. Y{\i}ld{\i}r{\i}m and M.~Todd}, {\em Sensitivity analysis in linear
  programming and semidefinite programming using interior-point methods},
  Mathematical Programming, 90 (2001), pp.~229--261.

\end{thebibliography}
\end{document}